\documentclass{article}
\usepackage{amsmath}
\usepackage{amsthm}
\usepackage{amssymb}
\usepackage{hyperref}
\usepackage{verbatim}
\usepackage{graphicx}
\usepackage[usenames,dvipsnames,svgnames,table]{xcolor}
\usepackage{mathrsfs}  

\usepackage[hmargin=3cm,vmargin=3.5cm]{geometry}
\def\a{\alpha}
\def\b{\beta}
\def\ga{\gamma}
\def\Ga{\Gamma}
\def\de{\delta}
\def\De{\Delta}
\def\ep{\epsilon}

\def\la{\lambda}

\def\si{\sigma}

\def\om{\omega}
\def\Om{\Omega}

\def\th{\theta}

\def\nab{\nabla}
\def\varep{\varepsilon}

\def\BB{{\cal B}}

\def\DD{{\cal D}}

\def\JJ{{\cal J}}

\def\DD{{\cal D}}

\def\RR{{\cal R}}

\newcommand{\N}[0]{\mathbb{N}}
\newcommand{\F}[0]{\mathbb{F}}
\newcommand{\R}[0]{\mathbb{R}}
\newcommand{\Z}[0]{\mathbb{Z}}

\newcommand{\T}[0]{\mathbb{T}}

\newcommand{\supp}{\mathrm{supp} \,}
\newcommand{\suppt}{\mathrm{supp}_t \,}

\newcommand{\fr}[2]{\frac{#1}{#2}}

\newcommand{\ALI}[1]{\begin{align*} #1 \end{align*}}
\newcommand{\tx}[1]{\mbox{#1}}

\newcommand{\lsm}[0]{\lesssim}

\newcommand{\wed}[0]{\wedge}
\newcommand{\pr}[0]{\partial}
\newcommand{\nb}{\nabla}

\newcommand{\co}[1]{\|#1\|_{C^0}}
\newcommand{\cda}[1]{\|#1\|_{\dot{C}^\a}}

\newcommand{\Ddt}[0]{\overline{D}_t}

\newcommand{\hxi}[0]{\widehat{\Xi}}          
\newcommand{\lhxi}[0]{\log \widehat{\Xi} \,}
\newcommand{\nhat}[0]{\widehat{N}}
\newcommand{\plhxi}[0]{(\log \widehat{\Xi})}
\newcommand{\va}[0]{\vec{a}}

\newcommand{\vcb}[0]{\vec{b}}

\DeclareMathAlphabet{\mathpzc}{OT1}{pzc}{m}{it}
\newcommand{\hh}[0]{\mathpzc{h}}
\newcommand{\bp}[0]{\bar{p}}

\newcommand{\hc}[0]{\widehat{C}}
\newcommand{\brk}[0]{(\bar{k})}
\newcommand{\bk}[0]{\bar{k}}

\newcommand{\VR}[0]{\mathring{V}}
\newcommand{\wOm}[0]{\widetilde{\Omega}}

\newcommand{\ever}[0]{\left( \fr{e_v}{e_R} \right)}
\newcommand{\wtld}[1]{\widetilde{#1}}

\newcommand{\ali}[1]{ \begin{align} #1 \end{align} }

\def\XXint#1#2#3{{\setbox0=\hbox{$#1{#2#3}{\int}$}
     \vcenter{\hbox{$#2#3$}}\kern-.5\wd0}}

\newcommand{\LCyl}{\hat{\Gamma}}
\newcommand{\ECyl}{\hat{C}}

\newtheorem{thm}{Theorem}
\newtheorem{lem}{Lemma}[section]

\newtheorem{prop}{Proposition}[section]

\theoremstyle{definition}
\newtheorem{defn}{Definition}[section]

\theoremstyle{remark}

\title{ On the Endpoint Regularity in Onsager's Conjecture }
\author{ Philip Isett\thanks{Department of Mathematics, University of Texas at Austin, Austin, TX  (\href{mailto:isett@math.utexas.edu}{isett@math.utexas.edu}) The work of P. Isett is supported by the National Science Foundation under Award No. DMS-1402370.} 
}
\date{ }

\begin{document}
\maketitle

\begin{abstract}
Onsager's conjecture states that the conservation of energy may fail for $3D$ incompressible Euler flows with H\"{o}lder regularity below $1/3$.  This conjecture was recently solved by the author, yet the endpoint case remains an interesting open question with further connections to turbulence theory.  In this work, we construct energy non-conserving solutions to the $3D$ incompressible Euler equations with space-time H\"{o}lder regularity converging to the critical exponent at small spatial scales and containing the entire range of exponents $[0,1/3)$.  

Our construction improves the author's previous result towards the endpoint case.  To obtain this improvement, we introduce a new method for optimizing the regularity that can be achieved by a convex integration scheme.  
A crucial point is to avoid power-losses in frequency in the estimates of the iteration.  This goal is achieved using localization techniques of \cite{IOnonpd} to modify the convex integration scheme.

We also prove results on general solutions at the critical regularity that may not conserve energy.  These include a theorem on intermittency stating roughly that energy dissipating solutions cannot have absolute structure functions satisfying the Kolmogorov-Obukhov scaling for any $p > 3$ if their singular supports have space-time Lebesgue measure zero.


\end{abstract}

\section{Introduction}

In this paper we consider the endpoint regularity in Onsager's conjecture for the incompressible Euler equations on $\R \times \T^3$, which we write in conservation form as
\ali{
\label{eq:eulerEqns} \tag{E}
\begin{split} 
\pr_t v^\ell + \nb_j(v^j v^\ell) + \nb^\ell p &= 0 \\
\nb_j v^j &= 0,
\end{split}
}
using the summation convention for summing repeated indices.  We are concerned mainly with weak solutions to the incompressible Euler equations, which are defined most generally as a locally square-integrable vector field $v$ (called the velocity field) and scalar function $p$ (called the pressure) that together satisfy \eqref{eq:eulerEqns} in the sense of distributions.

Onsager's conjecture states that for any H\"{o}lder exponent $\a < 1/3$ there exist periodic weak solutions to the $3D$ incompressible Euler equations
that belong to the H\"{o}lder class $v \in L_t^\infty C_x^\a$ and fail to conserve the total kinetic energy $\fr{1}{2}\int_{\T^3} |v(t,x)|^2 dx$.  The endpoint case of the conjecture is that the same statement should hold for $\a = 1/3$.  The above statements originate from Onsager's paper \cite{onsag} on the statistical theory of hydrodynamic turbulence, where Onsager postulated that dissipation of energy may occur in the absence of viscosity\footnote{A related and important open question is whether such energy dissipating solutions arise as zero viscosity limits of solutions to the Navier-Stokes equations.} 
through the mechanism of an energy cascade modeled by the incompressible Euler equations.  

Onsager's argument predicts that such energy dissipation should be possible for incompressible Euler flows with regularity exactly $1/3$.  Specifically, Onsager argued that the energy cascade occuring in a turbulent flow will result in an energy spectrum with a statistical power law consistent with exactly the (Besov or H\"{o}lder) regularity $1/3$ in the inertial range of frequencies, which agrees with the scaling laws of turbulence predicted by Kolmogorov's 1941 theory \cite{K41}.  (See also \cite{eyinkSreen, deLSzeOnsagSurv} for more detailed reviews of these statements and computations.)  On the other hand, Onsager asserted that conservation of energy must hold for every incompressible Euler flow $v \in L_t^\infty C_x^\a(I \times \T^3)$ with H\"{o}lder regularity $\a$ strictly above $1/3$.  A strengthening of this latter assertion was proved in \cite{CET} after initial work of \cite{eyink}, with the sharpest known result being that conservation of energy holds for weak solutions in the Besov class $v \in L_t^3 B_{3,c_0(\N)}^{1/3}$ \cite{ches}.  These results leave open the possibility that energy dissipation as considered by Onsager may be possible for solutions to incompressible Euler with exactly the critical regularity $1/3$ (e.g. for weak solutions in the class $v \in C_t C_x^{1/3}$), while the construction in \cite{eyink} of initial data with critical regularity and nonzero energy flux provides further evidence that dissipation of energy for weak solutions at the critical regularity should indeed exist.

Recently, the existence of weak solutions to incompressible Euler in the class $v \in L_t^\infty C_x^\a(\R \times \T^3)$ that fail to conserve energy has been established by the author for all $\a < 1/3$ in \cite{isettOnsag}.  The solutions are constructed using the method of convex integration, which was first introduced to the incompressible Euler equations by De Lellis and Sz\'{e}kelyhidi \cite{deLSzeIncl,deLSzeCts,deLSzeHoldCts} and was further developed towards improved partial results towards Onsager's conjecture in \cite{isett, buckDeLIsettSze, buckDeLSzeOnsCrit}.  The proof in \cite{isettOnsag} relies also on the use of Mikado flows introduced in \cite{danSze} to implement convex integration in combination with a new ``gluing approximation'' technique.


In the present work, we improve upon the result in \cite{isettOnsag} to construct solutions with borderline regularity that approaches the endpoint case at small length scales while failing to conserve energy.  Our main result is the following.
\begin{thm} \label{thm:mainThm} There exists $(v,p)$ a weak solution to the incompressible Euler equations that has non-empty, compact support in time on $\R \times \T^3$ and belongs to the class $v \in \bigcap_{\a < 1/3} C_{t,x}^\a$.  Moreover, one may arrange that $v$ also satisfies an estimate of the form
\ali{
|v(t,x + \De x) - v(t,x)| &\leq C |\De x|^{\fr{1}{3} - B \sqrt{\fr{\log\log |\De x|^{-1}}{\log |\De x|^{-1}} }} \label{eq:loglogborderbd}
}
for some constants $C$ and $B$ and for all $(t,x) \in \R \times \T^3$ and all $|\De x| \leq 10^{-2}$.  
\end{thm}
The theorem is significant for the following reasons.
\begin{itemize}
\item Theorem~\ref{thm:mainThm} demonstrates how close the method of convex integration can come to achieving the self-similar $L_t^\infty C_x^{1/3}$ regularity that corresponds to the Kolmogorov theory.
\item The theorem is the first result proved by convex integration that approaches the endpoint regularity and avoids the strictly positive gap in regularity from the endpoint faced by previous results.  In particular, we have that $v \in \bigcap_{\a < 1/3} C_{t,x}^\a$ rather than having regularity bounded strictly below the limiting exponent (i.e. $v \in C_{t,x}^{1/3 - \ep}$ for some $\ep > 0$).
\item The proof of Theorem~\ref{thm:mainThm} is based on a new algorithm that optimizes the regularity coming from a convex integration construction.  This algorithm also apparently identifies an evident barrier towards achieving the endpoint regularity exactly using the convex integration method.
\item The proof of Theorem~\ref{thm:mainThm} clarifies which techniques in the literature yield the sharpest regularity.
\end{itemize}
The constant $B$, which determines\footnote{Note that changing the value of $B$ in \eqref{eq:loglogborderbd} corresponds to an inequivalent norm.} the rate at which the regularity $1/3$ is approached at small scales, can be taken to be  $B = 2 \sqrt{2/3} + o(1)$, and this bound can be improved to $B = 4/3 + o(1)$ by combining our methods with the approach to the gluing approximation taken in \cite{BDLSVonsag} (see Sections~\ref{sec:iterateMainLem}-\ref{sec:improveBorderline} below).  For comparison, note that inequality \eqref{eq:loglogborderbd} with $O(\sqrt{\fr{\log\log |\De x|^{-1} }{\log |\De x|^{-1}}})$ replaced by $O(\fr{1}{\log |\De x|^{-1}})$ would correspond to exactly the endpoint regularity $L_t^\infty C_x^{1/3}$. 

The algorithm we develop to prove Theorem~\ref{thm:mainThm}, presented in Sections~\ref{sec:iterateMainLem}-\ref{sec:improveBorderline}, is the main novelty of our paper relating to the construction of solutions.  Later on we will discuss theorems that elaborate a general theory of endpoint solutions.  We expect that our algorithm can be adapted to give similar borderline regularity results in any known convex integration construction of H\"{o}lder continuous solutions in which power-losses of frequency in the estimates can be avoided.  In particular, 
the method is likely to generalize 
to isometric embeddings as in \cite{deLSzeC1iso} (but not \cite{deLInSze}), to nondegenerate active scalar equations \cite{isettVicol}, to the $2D$ Monge-Amp{\`e}re equation \cite{lewPak2015convex}, and to the SQG equation \cite{buckShkVicSQG}, 
and in these cases the $\log \log |\De x|^{-1}$ term appearing in \eqref{eq:loglogborderbd} should be replaced by a large constant.  It is hopeful that our algorithm for optimizing the regularity may also be useful for potential applications to simulating convex integration solutions.

To achieve solutions with borderline regularity, it is necessary that the proof avoids losses of powers of the frequency in the estimates of the iteration scheme.  An important point in this regard is that the approach to the gluing construction taken in \cite{isettOnsag} obtains estimates that lose only a power of the logarithm of the frequency.  These estimates require extending the timescale of the gluing beyond the standard timescale in the local existence theory for incompressible Euler, which would be inversely proportional to some $C^\a$ norm of the initial velocity gradient. (We note in contrast that the approach taken in \cite{BDLSVonsag} leads to power-losses in the frequency at several points in the proof.  These occur both in the gluing and convex integration in parts of the proof where local well-posedness theory, Schauder estimates and Calder\'{o}n-Zygmund commutator estimates are employed.)  Still there is one point in the proof in \cite{isettOnsag}, which occurs during the convex integration step, where one encounters a power-loss in frequency, and it is necessary to modify the convex integration part of the proof to obtain our borderline result.

To avoid this power-loss, we adapt the strategy of \cite{IOnonpd} for localizing the convex integration method, which relies on two main modifications to the construction to gain the necessary estimate.  The first point is to modify the construction using waves that are localized to small length scales and are each forced to obey the conservation of angular momentum in addition to the conservation of linear momentum.  The second point is to make use of the family of operators developed in \cite{IOnonpd} that give compactly supported, symmetric solutions to the divergence equation when the necessary conditions for solving the symmetric divergence equation are satisfied.  In combination, these modifications  allow one to avoid the power-loss in frequency that had been present in \cite{isett} while enabling the authors to extend previous work of \cite{isett} on $(1/5-\ep)$-H\"{o}lder Euler flows to the nonperiodic setting of $\R \times \R^3$.  Here we adapt these ideas to the present scheme to achieve an analogous improvement in our bounds.  We note that it is important for this gain that we rely on the approach to the nonstationary phase estimate based on a parametrix and nonlinear phase functions introduced in \cite{isett}.

Obtaining the endpoint case of Onsager's conjecture will require further new ideas, and it is of interest to study the behavior of potential energy non-conserving solutions with endpoint regularity and possible approaches to constructing them.  A convex integration approach to the endpoint regularity would be possible if something sufficiently close to an ``ideal'' Main Lemma can be proven where one has neither logarithmic nor power-losses in the frequency and the constant in the frequency growth is equal to $\hc = 1$ (as in a remark of \cite{IOnonpd}) or approaches $\hc = 1$ asymptotically at a suitable rate that can be gleaned from examining our algorithm.  
Such a construction appears to be presently out of reach; however, it may be considered favorable that convex integration constructions are able in general to yield solutions whose singularities occupy regions of space with positive volume.  As the following theorem demonstrates, singularities with positive Lebesgue measure are necessary for any energy non-conserving solution with critical regularity to exist provided the integrability exponent for this regularity is greater than $3$. 
\begin{thm}[Intermittency Theorem] \label{prop:singPos} A weak solution $(v,p)$ to incompressible Euler on $I \times \T^d$ or $I \times \R^d$ that dissipates or otherwise fails to conserve energy cannot belong to an endpoint class $v \in  L_t^r B_{r,\infty}^{1/3} \cap L_{t,x}^2$ with an integrability exponent $r > 3$ if its singular support has space-time Lebesgue measure zero.
\end{thm}
Theorem~\ref{prop:singPos} has a special significance in terms of intermittent scaling exponents in turbulence.  The K41 theory \cite{K41} predicts a scaling law of the form $\langle |v(x+\De x) - v(x)|^p \rangle^{1/p} \sim |\De x|^{1/3}$ for absolute structure functions (the Kolmogorov-Obukhov law), which mathematically corresponds to $B_{p,\infty}^{1/3}$ control of the velocity field.  The idea that ``intermittency'' (deviations from self-similarity and homogeneity) in the energy dissipation and singular structure of turbulence can lead to the failure of this scaling law for $p \neq 3$ was first attributed to Landau by Kolmogorov in the 1940's (see \cite[Section 5]{frisch1991global}).  Moreover, experimental studies have found evidence of such intermittency in the energy dissipation of turbulent flows accompanied by deviations from the Kolmogorov-Obukhov law arising from a multifractal structure \cite{meneveau1987multifractal,meneveau1990joint,meneveau1991multifractal}.  Theorem~\ref{prop:singPos} and its proof provide a rigorous sense in which lower dimensional singularities or energy dissipation in fact logically imply deviations from the Kolmogorov-Obukhov law, thus reinforcing the experimental findings.  

Theorem~\ref{prop:singPos} is a consequence of two facts that are also new remarks in the literature, which are a local version of the sharp energy conservation criterion in \cite{ches} and a result on integrability of the energy dissipation measure (see Theorems~\ref{prop:locOnsSingCrit} and \ref{prop:endIntDissMsr} below).  One would most likely expect that energy non-conserving solutions exist for the entire spectrum of endpoint spaces above, including the endpoint case of $L_t^\infty C_x^{1/3}$.  
For a more precise formulation of Theorem~\ref{prop:singPos} we refer to Section~\ref{sec:critSing}.  We also note the works of \cite{chesShv, luoShv2dHomog,luoShv2DhomogAdd, shvHomog} for further mathematical results related to intermittency.

In addition to having the endpoint regularity, Onsager's paper \cite{onsag} describes Euler flows that furthermore have decreasing kinetic energy.  Related to this point, we state the following Theorem.


\begin{thm} \label{prop:energyc1} If $(v, p)$ are a weak solution to \eqref{eq:eulerEqns} on $I \times \T^d$, $d \geq 2$ with $v \in C_t C_x^{1/3}$ (or more generally with $v \in C_t B_{3,\infty}^{1/3}$) then the total kinetic energy $e(t) = \int_{\T^d} \fr{|v(t,x)|^2}{2} dx$ is $C^1$ in time.
\end{thm}
Theorem~\ref{prop:energyc1} implies that the task of finding an energy dissipating solution in the class $v \in C_t C_x^{1/3}$ can be reduced to finding any example of a solution in this class that fails to satisfy energy conservation.  Such a solution would have total kinetic energy that is either strictly increasing or strictly decreasing on some open interval of time.  After possibly reversing time one obtains a solution with a decreasing energy profile on an open interval.  For $\a < 1/3$, the existence of energy-dissipating solutions in $C_tC_x^\a$ was proven recently in \cite{BDLSVonsag} by introducing an additional idea in the convex integration part of the proof to prescribe the energy profile of the solutions.  We expect that this technique\footnote{A related technical point is that the approach to prescribing the energy profile in \cite{BDLSVonsag} involves requiring the stress tensor $R^{j\ell}$ to be trace-free in addition to being symmetric.  It is also possible to prescribe the energy profile without imposing the trace-free requirement on $R^{j\ell}$; see \cite{IOnonpd,IO-presEn}. } should be possible to extend to the class described by \eqref{eq:loglogborderbd} for example by modifying the statement of our Main Lemma in a way similar to the analysis in \cite{IOnonpd,IO-presEn}.

The proof of Theorem~\ref{prop:energyc1}, presented in Section~\ref{sec:regKineticCrit} below, 
suggests that the failure of energy conservation for solutions in the critical space $v \in C_tC_x^{1/3}$ should be very common.  
The proof reduces the existence of an energy-dissipating solution to solving the Euler equations with appropriate initial data in the desired critical space for a short time.  However, one must be cautious that the Euler equations are ill-posed in $C_t C_x^\a$ or in $C_t B_{3,\infty}^{\a}$ for all $\a < 1$ as has been shown in \cite{chesShvIllp,bardTiti}, which presents a significant difficulty for constructing solutions in these spaces.  Complementing these negative results, our proof of Theorem~\ref{prop:energyc1} yields as a byproduct a necessary condition for a given divergence free vector field to be the initial datum of a solution in the class $v \in C_t B_{3,\infty}^{1/3}$.

With regard to energy dissipation and solving the Cauchy problem, a natural question is whether there is a simple, Onsager-critical function space in which {\it local} dissipation of energy is guaranteed along with total kinetic energy dissipation if one can solve the Cauchy problem in that space with the appropriate initial data.  
A simple criterion of this type is provided in Theorem~\ref{thm:localDissstable} of Section~\ref{sec:locDissstable}.  The regularity condition imposed to maintain local dissipation in this criterion is notably stronger than that assumed to control total kinetic energy in Theorem~\ref{prop:energyc1}, as our Theorem~\ref{thm:localDissstable} involves solutions in the function space $C_t C_x^{1/3}$ rather than assuming only Besov regularity in space.  

We now summarize the organization of the paper and the proof of our borderline result, Theorem~\ref{thm:mainThm}.  The general theory of endpoint solutions, including Theorems~\ref{prop:singPos}-\ref{thm:localDissstable}, is contained in Sections~\ref{sec:regKineticCrit}-\ref{sec:locDissstable}.  
We then summarize notation for the main body of the paper in Section~\ref{sec:notation}.  Sections~\ref{sec:mainLem}-\ref{sec:solveSymmdiv} contain the Main Lemma of the paper and our modification of the convex integration construction of \cite{isettOnsag}.  These sections assume familiarity with the convex integration construction in \cite{isettOnsag}.  Section~\ref{sec:iterateMainLem} explains the proof of Theorem~\ref{thm:mainThm} using the Main Lemma, and presents our new method for optimizing the regularity in a general convex integration scheme.  Section~\ref{sec:improveBorderline} outlines how to combine our methods with the approach to the gluing approximation taken in \cite{BDLSVonsag} to improve the rate of convergence to the critical exponent in the estimate \eqref{eq:loglogborderbd}.

$ $

\noindent{\bf Acknowledgments}  The author is thankful to S.-J. Oh for helpful conversations related to the endpoint regularity.  This material is based upon work supported by the National Science Foundation under DMS-1402370, DMS-1700312 and DMS-2055019.










\section{Regularity of kinetic energy at the critical exponent} \label{sec:regKineticCrit}
We start with a proof of Theorem~\ref{prop:energyc1} on the $C^1$ regularity of the kinetic energy profile for solutions of class $C_t B_{3,\infty}^{1/3}$.  In the next Section we prove Theorem~\ref{prop:singPos}.  We will use the summation convention for summing repeated upper and lower spatial indices, so that $v^\ell v_\ell = |v|^2$ and $\nb_\ell v^\ell = \mbox{div } v$.

The proof of Theorem~\ref{prop:energyc1} is an extension of the argument of \cite{CET} for proving energy conservation for weak solutions in the class $v \in L_t^3 B_{3,\infty}^{1/3+\ep}$ and of a remark in \cite{isett2} on the endpoint case.  Namely, suppose that $(v,p)$ is a weak solution to \eqref{eq:eulerEqns} with velocity of class $v \in C_t B_{3,\infty}^{1/3}(I \times \T^d)$, $d \geq 2$, $I$ an open interval.  Let $\eta_\ep$ be a standard mollifier in $\R^d$ at length scale $\ep$, and let $v_\ep^\ell = \eta_\ep \ast v^\ell$ denote the mollification of $v$ in the spatial variables.  Then, as in \cite{CET}, one has (using $v \in C_t L_x^2$) that
\ali{
\begin{split} \label{eq:CETconverge}
\fr{d}{dt} \int_{\T^d} \fr{|v|^2(t,x)}{2} dx &= \lim_{\ep \to 0} \fr{d}{dt} \int_{\T^d} \fr{|v_\ep|^2(t,x)}{2} dx = - \lim_{\ep \to 0} \int_{\T^d} \nb_j (v_\ep)_\ell R_\ep^{j\ell}(t,x) dx 
\end{split}\\
R_\ep^{j\ell}(t,x) &:= v_\ep^j(t,x) v_\ep^\ell(t,x) - \eta_\ep \ast(v^j v^\ell)(t,x),
}
where the convergence in \eqref{eq:CETconverge} holds in $\DD'(I)$.  (See \cite[Proof of Theorem 2.2]{isettOh} for a detailed presentation of this point.)  The rightmost term in \eqref{eq:CETconverge} gives rise to the family of trilinear forms $T_\ep[v, v, v](t) := \int_{\T^d} \nb_j (v_\ep)_\ell R_\ep^{j\ell}(t,x) dx$ that satisfy the following bound uniformly in $\ep$
\ali{
|T_\ep[u,v,w]|(t) &\lsm \| u(t,\cdot) \|_{B_{3,\infty}^{1/3}} \| v(t,\cdot)  \|_{B_{3,\infty}^{1/3}} \| w(t,\cdot) \|_{B_{3,\infty}^{1/3}} \label{eq:trilinBd}
}
by the commutator estimate of \cite{CET}.  Using \eqref{eq:trilinBd}, we have that the family of functions $T_\ep[v,v,v](t)$ are both uniformly bounded and equicontinuous on every compact subinterval of $I$, as they satisfy
\ali{
|T_\ep[v,v,v](t) - T_\ep[v,v,v](t_0)| &\lsm \| v(t,\cdot) - v(t_0,\cdot) \|_{B_{3,\infty}^{1/3}} \| v \|_{C_t B_{3,\infty}^{1/3}}^2,
}
and their moduli of continuity can therefore be bounded uniformly in $\ep$ in terms of the modulus of continuity of $v(t,\cdot)$ into $B_{3,\infty}^{1/3}(\T^d)$ and local bounds for $\| v(t,\cdot) \|_{B_{3,\infty}^{1/3}}$.  Consequently, the convergence in \eqref{eq:CETconverge} is actually uniform in $t$ on every open interval $J$ with compact closure in $I$, as the weak limit in $\DD'(J)$, which is unique, must also be achieved uniformly along subsequences by Arzel\`{a}-Ascoli.  (If the convergence were not uniform, there would exist a subsequence converging uniformly to a continuous function different from \eqref{eq:CETconverge}, which contradicts the weak convergence.)  The energy flux in \eqref{eq:CETconverge}, a priori in $\DD'(I)$, is thus continuous in $t$ on $I$, and the kinetic energy profile is therefore $C^1$ in $t$ on $I$.  

Note that one would typically expect the energy flux given by the right hand side of \eqref{eq:CETconverge} to be nonzero at any given time $t_0$ for a vector field with $v(t_0,\cdot) \in C_x^{1/3}$, as examples of divergence free initial data $v_0(x) \in C^{1/3}$ for which this limit can be positive are given in \cite{eyink, ches}.

We note also that our argument provides a necessary condition for a vector field $v_0(x) \in B_{3,\infty}^{1/3}$ to be realized as the initial datum of an Euler flow in the class $v \in C_t B_{3,\infty}^{1/3}$, which is that the limit $\lim_{\ep \to 0} T_\ep[v_0, v_0, v_0]$ on the right hand side of \eqref{eq:CETconverge} must exist and must also be independent of the chosen mollifying kernel $\eta_\ep$ so that the instantaneous rate of energy dissipation is well-defined at time $0$.

We now turn to the proof of Theorem~\ref{prop:singPos}.


\section{Singularities of dissipative solutions with critical regularity} \label{sec:critSing}
We now establish Theorem~\ref{prop:singPos} on the necessity of positive measure singularities Onsager critical solutions with integrability exponent $p > 3$ that do not conserve energy, which is an immediate consequence of Theorems~\ref{prop:locOnsSingCrit} and \ref{prop:endIntDissMsr} below.  Both theorems are stated in terms of Besov spaces whose basic properties we recall within the proofs.  We state the first Theorem~\ref{prop:locOnsSingCrit} in a sharp, critical space to make clear the severity of the singularity that is implicitly discussed in Theorem~\ref{prop:singPos}.

\begin{thm} \label{prop:locOnsSingCrit} Let $(v,p)$ be a weak solution to the incompressible Euler equations of class $v \in L_{t,x}^3$ on $I \times \T^d$ or $I \times \R^d$, $I$ an open interval.  Then the distribution
$-D[v,p] := \pr_t \left(\fr{|v|^2}{2} \right) + \nb_j \left( \left(\fr{|v|^2}{2} + p\right) v^j\right)$
has support contained in the singular support of $v$ relative to the critical space $L_t^3 B_{3,c_0(\N)}^{1/3}$.  
\end{thm}
Here we define the singular support of $v$ relative to the space $L_t^3 B_{3,c_0(\N)}^{1/3}$ to be the complement of those points $q = (t,x)$ for which there exists an open neighborhood $O_q$ of $q$ on which $v$ is represented by a distribution of class $L_t^3 B_{3,c_0(\N)}^{1/3}$.  
We recall the standard characterization of the $B_{r,\infty}^{1/3}$ norm of a vector field on an open set $\Om$ in $\R^d$, which is given by $\| v \|_{L^r(\Om)} + \sup_{h \in \R^d \setminus \{0\}} |h|^{-1/3} \| v(\cdot - h) - v(\cdot) \|_{L_x^r(\Om \cap (\Om + h))} $, and we also recall that  $C^\infty(\Om)$ is dense in $B_{r,c_0(\N)}^{1/3}(\Om)$ with respect to the $B_{r,\infty}^{1/3}$ norm.  
It is clear that the singular support of $v$ relative to $L_t^3 B_{3,c_0(\N)}^{1/3}$ is a subset of the usual singular support of $v$ as a distribution.  Related restrictions on the support of $D[v,p]$ under different hypotheses and with different proofs are given in \cite[Theorem 4.3]{chesShv}, \cite[Theorem 1]{drivasNguyen2018onsager} and \cite[Theorem 3.1]{bardTitiWied}.

Our second theorem asserts that weak solutions of class $v \in L_t^r B_{r,\infty}^{1/3}$ for integrability exponents $r > 3$ possess integrability for their corresponding energy dissipation measure $D[v,p]$.  The assumptions are given in a way that is sufficient for our application to proving Theorem~\ref{prop:singPos}.  
\begin{thm} \label{prop:endIntDissMsr} Let $(v,p)$ be a weak solution to incompressible Euler of class $v \in L_t^r B_{r,\infty}^{1/3}$ for some $r \geq 3$ on $I \times \T^d$ or $I \times \R^d$, $I$ an open interval.  Then the distribution $D[v,p]$ above is a (signed) measure.  If furthermore $r > 3$, this measure is absolutely continuous with respect to the Lebesgue measure, and its Radon-Nikodym derivative is of class $D[v,p] \in L_{t,x}^{r/3}$.
\end{thm}
It will be clear that the proof of Theorem~\ref{prop:endIntDissMsr} does not give absolute continuity in the case $r = 3$.  For example, the proof would apply to many other equations such as Burgers', where shock solutions give examples of $L_t^\infty B_{3,\infty}^{1/3}$ solutions for which the corresponding energy dissipation measure is not absolutely continuous.  There also exist time-independent divergence free vector fields demonstrating that our approach would not yield absolute continuity in the $r = 3$ case\footnote{R. Shvydkoy, Personal Communication.}.

\begin{proof}[Proof of Theorem~\ref{prop:singPos}]
Let us observe now that Theorem~\ref{prop:singPos} follows from Theorems~\ref{prop:locOnsSingCrit} and \ref{prop:endIntDissMsr}, focusing on the case of $I \times \R^d$.  Namely, if a weak solution $(v,p)$ is of class $v \in L_{t,x}^3 \cap L_{t,x}^2$ and does not conserve kinetic energy (meaning that the distribution $e(t) := \fr{1}{2} \int_{\R^d} |v|^2(t,x) dx$ is not a constant), then the distribution $D[v,p]$ is well defined and cannot be the $0$ distribution.  This statement can be checked by verifying that, for any test function $\psi \in C_c^\infty(I)$, one has by dominated convergence that
\ALI{
\left\langle \psi(t),  e'(t) \right\rangle_{\DD'(I)} & = \lim_{R \to \infty} \langle \psi(t) \chi_R(x), - D[v,p] \rangle_{\DD'(I \times \R^d)} \\
:= - \int_I \psi'(t) e(t) dt &=  - \lim_{R \to \infty} \int_{I \times \R^d} \left[ \psi'(t) \chi_R(x) \fr{|v|^2}{2} + \psi(t) \nb_j \chi_R(x) \left( \fr{|v|^2}{2} + p\right) v^j \right] dt dx,
}
where $\chi_R(x) = \chi(\fr{x}{R})$ is a rescaled bump function that is equal to $1$ in a growing neighborhood of the origin that encompasses the whole space as $R \to \infty$.  We use here that $\left( \fr{|v|^2}{2} + p\right) v^j$ and $\fr{|v|^2}{2}$ are both in $L_{t,x}^1(I \times \R^d)$ as $v \in L_{t,x}^2 \cap L_{t,x}^3$ and $p = \De^{-1} \nb_j \nb_\ell(v^j v^\ell) \in L_{t,x}^{3/2}$ by Calder\'{o}n-Zygmund theory\footnote{The case of $\T^d$ appears to be less standard than the $\R^d$ case but can be deduced from the $\R^d$ case using the local Calder\'{o}n-Zygmund theory in $\R^d$ as in  \cite{wangCZloc}.  See e.g. \cite[Proof of Theorem 6.2]{isettAOMS}.}, which implies that $\De^{-1} \nb_j \nb_\ell$ acts as a bounded operator on $L_{t,x}^{3/2}$ mapping two-tensors to scalars.  In fact the weaker condition $(1 + |x|)^{-1}\left( \fr{|v|^2}{2} + p\right) v^j \in L_{t,x}^1$ suffices for this proof.

For a solution of class $v \in L_t^r B_{r,\infty}^{1/3}$ with $r >3$, we have by Theorem~\ref{prop:endIntDissMsr} that $D[v,p]$ is of class $L_{t,x}^{r/3}$.  For $D[v,p]$ to be nonzero, the support of $D[v,p]$ as a distribution must then occupy a closed set with positive Lebesgue measure.  From Theorem~\ref{prop:locOnsSingCrit}, the nontrivial support of $D[v,p]$ gives a lower bound for the singular support of $v$ as a distribution, which implies Theorem~\ref{prop:singPos}. 
\end{proof}
We now prove Theorem~\ref{prop:locOnsSingCrit} along with Theorem~\ref{prop:endIntDissMsr}.  The proof is a local version of the energy conservation criteria of \cite{CET, ches}.  The observation that the proof of energy conservation in \cite{CET} can be localized is originally due to \cite{duchonRobert} and has recently been of use to several authors in the context of bounded domains \cite{bardTitiBdd,drivasNguyen2018onsager,bardTitiWied}.  Some issues that are not central to our goals here have been avoided as our hypotheses suffice to guarantee $p = \De^{-1} \nb_j \nb_\ell(v^j v^\ell) \in L_{t,x}^{3/2}$.  The norms and functions space in what follows refer to the entire space $I \times \T^d$ or $I \times \R^d$ unless otherwise stated.  We will focus on the $\R^d$ cases in what follows as the results for $\T^d$ follow from the same proofs.    
\begin{proof}[Proof of Theorems~\ref{prop:endIntDissMsr} and \ref{prop:locOnsSingCrit}]  Let $(v,p)$ be a weak solution of class $v \in L_t^r B_{r,\infty}^{1/3} \cap L_{t,x}^2$ for some $r \geq 3$.  Then $v \in L_{t,x}^r \cap L_{t,x}^2$ and $p = \De^{-1} \nb_j \nb_\ell(v^j v^\ell) \in L_{t,x}^{r/2}$ by Calder\'{o}n-Zygmund theory as before.  The key formula we use is the analogue of the \cite{duchonRobert} formula involving the commutator of \cite{CET}:
\ali{
-D[v,p] = \pr_t \left(\fr{|v|^2}{2}\right) \nb_j \left[ \left(\fr{|v|^2}{2} + p \right) v^j \right] &= \lim_{\ep \to 0} \nb_j v_{\ep\ell} R_\ep^{j\ell} \label{eq:dissipRateLim} \\
R_\ep^{j\ell} &= \eta_\ep \ast(v^j v^\ell) - v_\ep^j v_\ep^\ell, \label{eq:stress}
}
where $v_\ep^\ell = \eta_\ep \ast v^\ell$ is a standard mollification of $v^\ell$ in the spatial variables at length scale $\ep$, and the limit \eqref{eq:dissipRateLim} holds for any fixed test function on $I \times \R^d$ or $I \times \T^d$.

We first prove Theorems~\ref{prop:endIntDissMsr} and \ref{prop:locOnsSingCrit} assuming \eqref{eq:dissipRateLim}.  One has by H\"{o}lder's inequality with $\fr{3}{r} = \fr{1}{r} + \fr{2}{r}$ and the commutator estimates of \cite{CET} the following bound uniformly in $\ep$ 
\ali{
\| \nb_j v_{\ep\ell} R_\ep^{j\ell} \|_{L_{t,x}^{r/3}} &\leq  \| \nb_j v_{\ep\ell} \|_{L_{t,x}^r} \|R_\ep^{j\ell} \|_{L_{t,x}^{r/2}} \notag \\
&\lsm ( \ep^{-1 + 1/3} \| v \|_{L_t^r B_{r,\infty}^{1/3}} ) \|R_\ep^{j\ell} \|_{L_{t,x}^{r/2}} \notag \\
&\lsm ( \ep^{-1 + 1/3} \| v \|_{L_t^r B_{r,\infty}^{1/3}} ) \ep^{2/3} \| v \|_{L_t^r B_{r,\infty}^{1/3}}^2 \notag \\
\| \nb_j v_{\ep\ell} R_\ep^{j\ell} \|_{L_{t,x}^{r/3}} &\lsm \| v \|_{L_t^r B_{r,\infty}^{1/3}}^3. \label{eq:unifBd}
}
The sequence $\nb_j v_{\ep\ell} R_\ep^{j\ell}$ is therefore uniformly bounded in $L_{t,x}^{r/3}$ independent of $\ep > 0$.

As a consequence, using $r \geq 3$, the weak limit $D[v,p] = \lim_{\ep \to 0} \nb_j v_{\ep\ell} R_\ep^{j\ell}$ is a Radon measure.  That is, by \eqref{eq:unifBd} and H\"{o}lder's inequality (with the characteristic function of $K$ as one of the factors), for any compact set $K$ and any test function $\phi(t,x)$ supported in $K$, one has 
\ALI{|\langle \phi, D[v,p]\rangle_{\DD'(I\times\R^d)}| \leq C_K \| \phi \|_{C^0} \| v \|_{L_t^r B_{r,\infty}^{1/3}}^3.}

Moreover, for $r > 3$, the measure $D[v,p]$ is absolutely continuous with density function in $L_{t,x}^{r/3}$ by the duality characterization of the latter space, thus confirming Theorem~\ref{prop:endIntDissMsr}.  Namely, if $s \in (1,\infty)$ is the dual exponent with $\fr{1}{s} + \fr{3}{r} = 1$, we have $|\langle \phi, D[v,p]\rangle_{\DD'(I\times\R^d)}| \leq C \| \phi \|_{L_{t,x}^{s}} \| v \|_{L_t^r B_{r,\infty}^{1/3}}^3$.  From the density of test functions in $L_{t,x}^s$, we have that $D[v,p]$ is in the dual of $L_{t,x}^s$, which is the space $L_{t,x}^{r/3}$.

The proof of Theorem~\ref{prop:locOnsSingCrit} is more subtle as the statement concerns the function space $L_t^3 B^{1/3}_{3,c_0(\N)}$ and is more local in nature.  In particular, our approach is local as compared to the Fourier-analytic approach of \cite{ches}; the details in the presentation below are similar to those of \cite{IOheat}.

Let $v \in L_{t,x}^3$ be a weak solution, so that $p \in L_{t,x}^{3/2}$, and let $q$ be a point in the complement of the singular support of $v$ relative to $L_t^3 B_{3,c_0(\N)}^{1/3}$.  That is, there is an open neighborhood of $q$ that can be taken to have the form $J \times B_q$ with $J$ a finite open subinterval of $I$ and $B_q$ a spatial ball such that $v \in L_t^3 B_{3,c_0(\N)}^{1/3}(J \times B_q)$.  Let $\phi \in C_c^\infty(J \times B_q)$ be a fixed test function and $B_q' \subseteq B_q$ be a smaller spatial ball such that $\supp \phi \subseteq J \times B_q'$.  From \eqref{eq:dissipRateLim}, we have 
\ALI{
\langle \phi, - D[v,p] \rangle &= \lim_{\ep \to 0} \int_{J} \int_{B_q'} \phi(t,x) \nb_j v_{\ep\ell} R_\ep^{j\ell} dx dt
}
where by assumption $v \in L_t^3 B_{3,c_0(\N)}^{1/3}(J \times B_q)$.  Then as in the proof of \eqref{eq:unifBd} one has that 
\ali{
|\langle \phi, - D[v,p] \rangle| &\leq \limsup_{\ep \to 0} \| \phi \|_{C^0} \int_{J}  \| \nb v_\ep(t,\cdot) \|_{L^3(B_q')} \| R_\ep^{j\ell}(t,\cdot) \|_{L^{3/2}(B_q')} dt, \label{eq:beforeDomConverge} 
}
and that the $dt$ integrand is bounded uniformly in $\ep$ by $C \| v(t,\cdot) \|_{B_{3,\infty}^{1/3}(B_q)}^3$, which is integrable over $J$.  Moreover, for almost every $t \in J$, one has that $v(t,\cdot) \in B_{3,c_0(\N)}^{1/3}$ belongs to the closure of $C^\infty(B_q)$ in the $B_{3,\infty}^{1/3}$ norm.  For each such $t$, the improved bound $\limsup_{\ep \to 0} \ep^{1 - 1/3} \| \nb v_\ep(t,\cdot) \|_{L^3(B_q')} = 0$ holds, as can be seen by a smooth approximation argument.  Combined with $\| R_\ep^{j\ell}(t,\cdot) \|_{L^{3/2}(B_q')} \leq C_t \, \ep^{2/3} $ on the same set of $t$, we have the convergence to $0$ for almost every $t$ in \eqref{eq:beforeDomConverge}, which implies the limit in \eqref{eq:beforeDomConverge} is zero by the Lebesgue dominated convergence theorem.

The last remaining point is to justify the limit in \eqref{eq:dissipRateLim} for any fixed test function, which we prove using the definition of a weak solution following details similar to \cite{IOheat}.  Let $(v,p)$ be a weak solution of class $v \in L_{t,x}^3$, so that $p \in L_{t,x}^{3/2}$ on $I \times \R^d$ as before.  
Let $\phi \in C_c^\infty$ be a test function on $I \times \R^d$ and $V_\phi$ be an open set with compact closure in $I \times \R^d$ that contains $\supp \phi$.  Let $\eta_\ep(h) = \ep^{-d} \eta(\fr{h}{\ep})$ and $\zeta_\de(\tau) = \de^{-1} \zeta(\fr{\tau}{\de})$ be even mollifying kernels in the space and time variables respectively supported in $\supp \eta_\ep \subseteq B_\ep(0)$ in $\R^d$ and $\supp \zeta_\de \subseteq B_\de(0)$ in $\R$.  Define $\eta_{\ep\de}(\tau,h) = \zeta_\de(\tau) \eta_\ep(h)$ and the vector field $\om_{\ep \de}^\ell = \eta_{\ep\de} \ast \left( \phi \, \eta_{\ep\de}\ast v^\ell \right)$, where the convolution is in both space and time.  We will write $\ast_x$ or $\ast_t$ to mean convolution in only the space or time variables.  Taking $\om_{\ep \de}^\ell$ as our test function in the weak formulation of Euler (i.e. multiplying the equation and integrating by parts) gives
\ALI{
- \int_{I \times \R^d}\left[ v^\ell \pr_t \eta_{\ep\de} \ast \left( \phi \, \eta_{\ep\de}\ast v_\ell \right) + v^j v^\ell \nb_j \eta_{\ep\de} \ast \left( \phi \, \eta_{\ep\de}\ast v_\ell \right) + p \nb^\ell \eta_{\ep\de} \ast \left(\phi \, \eta_{\ep\de}\ast v_\ell \right) \right] dx dt &= 0.
}
Using the self-adjointness of $\eta_{\ep\de} \ast$ and the divergence free property of $\eta_{\ep\de} \ast v^\ell$ one obtains
\ALI{
- \int_{I \times \R^d}\left[ \pr_t \phi(t,x) \fr{|\eta_{\ep\de} \ast v^\ell|^2}{2}  +  (v^j v^\ell) \, \eta_{\ep\de} \ast \nb_j[ \phi \, \eta_{\ep\de}\ast v_\ell ] +  p \, \eta_{\ep\de} \ast \left( \nb^\ell \phi \, \eta_{\ep\de}\ast v_\ell \right) \right] dx dt &= 0.
}
As $v \in L_{t,x}^3 \cap L_{t,x}^2(V_\phi)$ and $p \in L_{t,x}^{3/2}(V_\phi)$, we may safely let $\de \to 0$ at this point with $\ep > 0$ fixed using uniform in $\de$ boundedness of the convolution operators in the formula (including the operators $\nb_j \eta_{\ep\de} \ast$ that appear from the product rule) and the strong convergence of $\eta_{\ep \de} \ast v^\ell \to v_\ep^\ell := \eta_\ep \ast_x v^\ell$ in $L_{t,x}^2 \cap L_{t,x}^3(\supp \phi)$ for each fixed $\ep > 0$.  Taking the $\de \to 0$ limit, we may replace each appearance of $\eta_{\ep \de} \ast = \eta_\ep \ast_x [\zeta_\de \ast_t \cdot ]$ in the formula with $\eta_\ep \ast_x$, which we now write more simply as $\eta_\ep \ast := \eta_\ep \ast_x$.

Using the self-adjointness of $\eta_\ep \ast$ and the divergence free property of $v_\ep^\ell$, which are justified by the same limiting argument, one then obtains
\ALI{
- \int_{I \times \R^d}\left[ \pr_t \phi(t,x) \fr{|\eta_{\ep} \ast v^\ell|^2}{2} \right.&\left.+ \nb_j \phi(t,x) \left( \fr{|v_\ep|^2}{2} v_\ep^j + \eta_\ep \ast p \, v_\ep^j \right) \right] dx dt = \int_{I \times \R^d} \phi(t,x) \nb_j v_{\ep \ell} R_\ep^{j\ell} dx dt + Z_\ep, \\
Z_\ep &:= \int_{I \times \R^d} \nb_j \phi \, R_\ep^{j\ell} v_{\ep \ell}  \, dx dt.
}
Note that the left hand side of the first equation tends to exactly $\langle \phi, -D[v,p] \rangle_{\DD'(I \times \R^d)}$ as $\ep \to 0$, using that $v_\ep^\ell = \eta_\ep \ast v^\ell \to v^\ell$ in $L_{t,x}^3 \cap L_{t,x}^2(V_\phi)$ and that $p \in L_{t,x}^{3/2}$ again.  Thus formula \eqref{eq:dissipRateLim} will be proven once it is shown that $\lim_{\ep \to 0} Z_\ep = 0$.

To this end, write $R_\ep^{j\ell}$ in terms of bilinear operators $R_\ep^{j\ell} = B_\ep[v^j, v^\ell]$, where the operators $B_\ep$ are defined  for smooth $u^j, w^\ell$ by 
$B_\ep[u^j, w^\ell] := \eta_\ep \ast (u^j w^\ell) - \eta_\ep \ast u^j \eta_\ep \ast w^\ell$.  
One has then that $\| B_\ep[u, w] \|_{L_{t,x}^{3/2}(V_\phi)} \to 0$ as $\ep \to 0$ whenever $u^j, w^\ell$ are smooth vector fields on $I \times \R^d$, and that $\| B_\ep[u, w] \|_{L_{t,x}^{3/2}(V_\phi)} \leq C \| u \|_{L_{t,x}^3} \| w \|_{L_{t,x}^3(I\times \R^d)}$ uniformly in $\ep > 0$.  Combining these properties and using the density of smooth vector fields in $L_{t,x}^3(I\times \R^d)$, we obtain that $\| R_\ep^{j\ell} \|_{L_{t,x}^{3/2}(V_\phi)} \to 0$ as $\ep \to 0$, and $Z_\ep \to 0$ as well by applying H\"{o}lder with $v_\ep$ bounded in $L_{t,x}^3(V_\phi)$.
\end{proof}

\subsection{Stability of local energy dissipation in a critical class} \label{sec:locDissstable}
In this section we prove Theorem~\ref{thm:localDissstable}, which provides a simple function space criterion from which one can deduce local dissipation on an open interval of time from local dissipation at time $0$.

\begin{thm} \label{thm:localDissstable} Let $\bar{v}$ be a divergence free vector field of class $\bar{v} \in C^{1/3}(\T^d)$ for which the local energy dissipation is everywhere bounded by a strictly negative constant.  Then any weak solution $(v,p)$ of class $v \in C_t C^{1/3}(I \times \T^d)$ that obtains the initial data $\bar{v}$ must satisfy the local energy inequality $D[v,p] < 0$ on some open time interval containing $t = 0$.
\end{thm}
The precise condition on the initial data $\bar{v}$ will be specified in line \eqref{eq:localDissipIncondit} of the proof below.
\begin{proof}  Let $\bar{v}$ be as above and let $(v, p)$ be a weak solution to Euler of class $v \in C_t C^{1/3}(I \times \T^d)$ on an open interval of time containing $t = 0$ with initial data $\bar{v}(x)$.  Let $\tilde{I}$ be an open subinterval of $I$ containing $t = 0$, and let $\phi \in C_c^\infty(\tilde{I} \times \T^d)$ a non-negative test function supported in $t \in \tilde{I}$.  As in the previous sections, we have
\ALI{
\langle \phi, D[v,p] \rangle &= \lim_{\ep \to 0} \int_{\tilde{I}} \int_{\T^d} \phi(t,x) T_\ep[v](t,x) dx dt \\
T_\ep[v](t,x) &= \nb_j v_{\ep \ell} R_{\ep}^{j\ell}(t,x),
}
where $T_\ep$ is the trilinear form from Section~\ref{sec:regKineticCrit} and $D[v,p]$ is as in the previous sections.
We write
\ali{
\langle \phi, D[v,p] \rangle &= \lim_{\ep \to 0} \int_{\tilde{I}} \int_{\T^d} \phi(t,x) ( T_\ep[v](t,x) - T_\ep[v](0,x)) dx dt \notag \\ 
&+ \lim_{\ep \to 0} \int_{\tilde{I}} \int_{\T^d} \phi(t,x) T_\ep[v](0,x) dx dt \label{eq:signedTerm}
}
The precise assumption placed on the initial condition $\bar{v}$ is that
\ali{
\lim_{\ep \to 0} T_\ep[v](0, x) = \lim_{\ep \to 0} \nb_j \bar{v}_{\ep \ell} R_\ep^{j\ell}(0,x) \leq - \varep < 0 \label{eq:localDissipIncondit}
}
in the sense of distributions on $\T^d$ for some constant $\varep > 0$.  
Integrating  \eqref{eq:localDissipIncondit} against the non-negative test function $\tilde{\phi}(x) = \int_{\tilde{I}}  \phi(t,x) dt \in C^\infty(\T^d)$, we have that
\ALI{
\eqref{eq:signedTerm} \leq - \varep \int_{\T^d} \left[ \int_{\tilde{I}}  \phi(t,x) dt \right] dx.
}
For a sufficiently small time interval $\tilde{I}$, we can obtain the bound 
\ALI{
\| T_\ep[v](t,x) - T_\ep[v](0,x) \|_{L^\infty(\tilde{I} \times \T^d) } \leq \fr{\varep}{2} 
}
using that $v \in C_t C_x^{1/3}$ is continuous in time with values in $C^{1/3}$ and the commutator estimate of \cite{CET} to control the bilinear term.  Combining these estimates with the sign condition on $\phi$ gives 
\ALI{
\langle \phi, D[v,p] \rangle &\leq \| \phi \|_{L^1(\tilde{I} \times \T^d)} \fr{\varep}{2}  - \varep   \int_{\tilde{I}} \int_{\T^d} \phi(t,x) dx dt \\
\langle \phi, D[v,p] \rangle &\leq -\fr{\varep}{2} \int_{\tilde{I}} \int_{\T^d} \phi(t,x) dx dt
}
for all non-negative $\phi \in C_c^\infty(\tilde{I} \times \T^d)$.  This bound shows that $D[v,p] \leq -\varep/2 < 0$ as a distribution when restricted to $\tilde{I}\times \T^d$, which concludes the proof of Theorem~\ref{thm:localDissstable}.
\end{proof}

With Theorems~\ref{prop:singPos}-\ref{thm:localDissstable} now proven, 
we turn to the notation that will be used for the remainder of the paper and the proof of Theorem~\ref{thm:mainThm}.

\section{Notation} \label{sec:notation}
We will follow the same notational conventions as introduced in \cite[Section 2]{isettOnsag}.  In particular, multi-indices will be represented in vector notation.  For example, if $\va = (a_1, a_2, a_3)$ is a multi-index of order $|\va| = 3$, each $a_i \in \{ 1, 2, 3 \}$, then $\nb_{\va} = \nb_{a_1} \nb_{a_2} \nb_{a_3}$ denotes the corresponding third-order partial derivative operator.  We use $\suppt f$ to indicate the time support of a function $f$ with domain in $\R \times \T^3$ (i.e., the closed set of times for which $\{ t \} \times \T^3$ intersects the usual support).


We recall the definitions of an Euler-Reynolds flow and frequency-energy levels.
\begin{defn}\label{defn:euReynFlow}  A vector field $v^\ell : \R \times \T^3 \to \R^3$, function $p : \R \times \T^3 \to \R$ and symmetric tensor field $R^{j\ell} : \R \times \T^3 \to \R^3 \otimes \R^3$ satisfy the {\bf Euler-Reynolds equations} if the equations
\ALI{
\pr_t v^\ell + \nb_j(v^j v^\ell) + \nb^\ell p &=\nb_j R^{j\ell} \\
\nb_j v^j &= 0
}
hold on $\R \times \T^3$.  Any solution to the Euler-Reynolds equations $(v,p,R)$ is called an {\bf Euler-Reynolds flow}.  The symmetric tensor field $R^{j\ell}$ is called the {\bf stress} tensor.
\end{defn}
\begin{defn}\label{defn:frenlvls}  Let $(v, p, R)$ be a solution of the Euler-Reynolds equation, $\Xi \geq 3$ and $e_v \geq e_R > 0$ be positive numbers.  We say that $(v, p, R)$ have {\bf frequency-energy levels} bounded by $(\Xi, e_v, e_R)$ to order $L$ in $C^0$ if $v$ and $R$ are of class $C_t C_x^L$ and the following estimates hold
\ali{
\co{\nab_{\va} v} &\leq \Xi^{|\va|} e_v^{1/2}, \quad \tx{ for all } 1 \leq |\va| \leq L \label{eq:frEnVelocbds} \\
\co{\nab_{\va}R} &\leq \Xi^{|\va|} e_R, \quad \tx{ for all } 0 \leq |\va| \leq L.
}
Here $\nab$ refers only to derivatives in the spatial variables.
\end{defn}

\section{The Main Lemma} \label{sec:mainLem}
The first goal of the paper will be to improve on the Main Lemma in \cite{isettOnsag} so that we remove the need for a double-exponential growth of frequencies.  The Main Lemma of our paper states the following:
\begin{lem}[The Main Lemma] \label{lem:mainLem}  Let $L = 3$.  There exists constants $\hc, C_L$ such that the following holds.   Let $(v, p, R)$ be any solution of the Euler-Reynolds equation with frequency-energy levels bounded by $(\Xi, e_v, e_R)$ to order $L$ in $C^0$ and let $J$ be an open subinterval of $\R$ such that 
\ALI{
\suppt v \cup \suppt R &\subseteq J.
}
Define the parameter $\hxi = \Xi (e_v/e_R)^{1/2}$.  Let $N$ be any positive number obeying the condition
\ali{
N \geq (e_v/e_R)^{1/2}.  \label{eq:Nrestrict}
}
Then there exists a solution $(v_1, p_1, R_1)$ of Euler-Reynolds with frequency-energy levels bounded by
\ali{
(\Xi', e_v', e_R') &= \left(\hc N \Xi, \plhxi e_R, \plhxi^{5/2} \fr{e_v^{1/2} e_R^{1/2}}{N} \right) \label{eq:newFreqEn}
}
to order $L$ in $C^0$ such that 
\ali{
\suppt v_1 \cup \suppt R_1 &\subseteq N(J; \Xi^{-1} e_v^{-1/2}) \label{ct:suppCdn}
}
and such that the correction $V = v_1 - v$ obeys the estimate
\ali{
\co{V} &\leq C_L \plhxi^{1/2} e_R^{1/2}. \label{ineq:coBdV} 
}
\end{lem}
The crucial difference between the Main Lemma above as compared to \cite[Lemma 2.1]{isettOnsag} is that we do not require any lower bound of the form $N \geq \Xi^\eta$ for the frequency growth parameter $N$ in inequality \eqref{eq:Nrestrict}.  This difference enables us to avoid double-exponential growth of frequencies in constructing solutions as in \cite{IOnonpd}.  Likewise, the constants $\hc$ and $C_L$ in the estimates do not depend on such a parameter $\eta$.


We establish Lemma~\ref{lem:mainLem} by modifying the proof of the Convex Integration Lemma, \cite[Lemma 3.3]{isettOnsag}, as the proof of this Lemma contains the only step in which the assumption $N \geq \Xi^\eta$ is used.  
\section{The Improved Convex Integration Lemma}
As in \cite{isettOnsag}, we will establish Lemma~\ref{lem:mainLem} by combining a Gluing Approximation Lemma and a Convex Integration Lemma.  In Lemma~\ref{lem:glueLem} below, we summarize the result of combining the Regularization Lemma and the Gluing Approximation Lemma from \cite[Section 3]{isettOnsag}.  (Here we have renamed the Euler-Reynolds flow $(\tilde{v}, \tilde{p}, \wtld{R})$ in \cite[Lemma 3.2]{isettOnsag} to be $(v,p,R)$.)
\begin{lem}[Gluing Approximation Lemma] \label{lem:glueLem} There are absolute constants $C_1 \geq 2$ and $\de_0 \in (0, 1/25)$ such that the following holds.  Let $(v_0, p_0, R_0)$ be an Euler-Reynolds flow with frequency-energy levels bounded by $(\Xi, e_v, e_R)$ to order $3$ in $C^0$ such that $\suppt v_0 \cup \suppt R_0 \subseteq J$.  Define the parameters
\ali{
\nhat := (e_v/e_R)^{1/2}, \quad  \hxi := \nhat \Xi = (e_v/e_R)^{1/2} \Xi.
}
Then for any $0 < \de \leq \de_0$ there exist: a constant $C_\de \geq 1$, a constant $\th > 0$, a sequence of times $\{t(I)\}_{I \in \Z} \subseteq \R$ and an Euler-Reynolds flow $(v, p, R)$, $R = \sum_{I \in \Z} R_I$, that satisfy the support restrictions
\ali{
\suppt v \cup \suppt R &\subseteq N(J; 3^{-1}\Xi^{-1} e_v^{-1/2}) \label{ct:growth} \\
2^{-1} \de \plhxi^{-2} \Xi^{-1} e_v^{-1/2} \leq \th &\leq  \de \plhxi^{-2} \Xi^{-1} e_v^{-1/2} \label{ineq:thBound} \\
\suppt R_I &\subseteq \left[t(I) - \fr{\th}{2}, t(I) + \fr{\th}{2}\right] \label{eq:supptRI} \\
\bigcup_I \bigcup_{I' \neq I} [t(I) - \th, t(I) + \th] &\cap [t(I') - \th, t(I') + \th] = \emptyset \label{ct:disjointness}
}
and the following estimates
\ali{
\co{ v - v_0 } &\leq C_1 e_R^{1/2}  \\
\co{ \nab_{\va} v } &\leq C_1 \Xi^{|\va|} e_v^{1/2}, \quad |\va| = 1, 2, 3 \label{eq:newVelocBdGlue} \\
\sup_I \co{ \nab_{\va} R_I  } &\leq C_\de \nhat^{(|\va|-2)_+} \Xi^{|\va|} \lhxi  e_R, \quad |\va| = 0, 1,2, 3 \label{eq:newRIbdGlue} \\
\sup_I \co{ \nab_{\va} (\pr_t + v \cdot \nab) R_I } &\leq C_\de  \plhxi^3 \Xi e_v^{1/2} \Xi^{|\va|} e_R, \quad |\va| = 0, 1, 2. \label{eq:newDtRIbdGlue}
}
\end{lem}
Our improved Convex Integration Lemma may then be stated as follows.
\begin{lem}[The Convex Integration Lemma] \label{lem:convexInt} There exists an absolute constant $b_0$ such that for any $C_1, C_\de \geq 1$ and $\de > 0$ there is a constant $\tilde{C} = \tilde{C}_{\de, C_1, C_\de}$ for which the following holds.  Suppose $J$ is a subinterval of $\R$ and $(v, p, R)$ is an Euler-Reynolds flow, $R = \sum_I R_I$, that satisfy the conclusions \eqref{ct:growth}-\eqref{ct:disjointness} and \eqref{eq:newVelocBdGlue}-\eqref{eq:newRIbdGlue} of Lemma~\ref{lem:glueLem} for some $(\Xi, e_v, e_R)$, some $\th  > 0$ and some sequence of times $\{ t(I)\}_{I \in \Z} \subseteq \R$.  Suppose also that
\ali{
|\th| \co{ \nab v } &\leq b_0. \label{eq:b0bd}
}
Let $N \geq (e_v/e_R)^{1/2}$.  Then there is an Euler-Reynolds flow $(v_1, p_1, R_1)$ with frequency-energy levels in the sense of Definition~\ref{defn:frenlvls} bounded by
\ali{
(\Xi', e_v', e_R') &= \left(\tilde{C} N \Xi, \plhxi e_R, \plhxi^{5/2} \fr{e_v^{1/2} e_R^{1/2}}{N} \right) \label{eq:newFrEnLvls}
}
such that
\ali{
\suppt v_1 \cup \suppt R_1 &\subseteq N(J; \Xi^{-1} e_v^{-1/2}) \label{ct:suppCvxInt} \\
\co{ v_1 - v } &\leq \tilde{C} \plhxi^{1/2} e_R^{1/2}. \label{eq:cobdcorrect}
}
\end{lem}
 Lemma~\ref{lem:mainLem} now follows by combining Lemmas~\ref{lem:glueLem} and \ref{lem:convexInt} as explained in \cite[Section 3]{isettOnsag}.  (Here Lemma~\ref{lem:glueLem} is applied with $(v_0, p_0, R_0)$ taken to be the $(v,p,R)$ given in the assumptions of Lemma~\ref{lem:mainLem}.)  The only important difference in the present case is that we have removed the assumption $N \geq \Xi^\eta$ and the constants $\hc$ and $C_L$ (which can be set equal if desired) do not depend on $\eta$.

We now explain how to prove Lemma~\ref{lem:convexInt} by modifying the proof of \cite[Lemma 3.3]{isettOnsag}.

\section{Modifying the Convex Integration} \label{sec:modifyConvexInt}
We now proceed with the proof of Lemma~\ref{lem:convexInt}.  The construction will be based on the proof of \cite[Lemma 3.3]{isettOnsag} implementing convex integration with the Mikado flows of \cite{danSze}, but modified to adapt the localization strategy of \cite{IOnonpd} to our setting.  

Let $(v,p,R)$, $R = \sum_I R_I$ be given as in the assumptions of Lemma~\ref{lem:convexInt}, which are the conclusions of Lemma~\ref{lem:glueLem}.  We will use the symbol $\lsm$ to denote inequalities involving explicit constants that are allowed to depend on the parameters $C_1, \de$ and $C_\de$, but never on $(\Xi, e_v, e_R), N, \th, \hxi,$ etc.  

We obtain the new Euler-Reynolds flow $(v_1, p_1, R_1)$ of Lemma~\ref{lem:convexInt} by adding carefully designed corrections $v_1^\ell = v^\ell + V^\ell$, $p_1 = p + P$ to the velocity and pressure, and using the resulting equation for $(v_1, p_1)$ to construct the appropriate $R_1$.  The correction $V^\ell$ will be a sum of divergence free, high frequency vector fields indexed by a set $\JJ$ 
\ALI{
V^\ell &= \sum_{J \in {\cal J}} V_J^\ell, \qquad \nb_\ell V_J^\ell = 0, \qquad \mbox{ for all } J \in \JJ
}
The index $J \in \JJ$ will have several components $J = (I, J_1, J_2, J_3, f)$ that together specify the time interval and spatial location in which $V_J$ will be supported as well as the direction in which $V_J$ takes values.  Specifically, we choose an even integer $\Pi \in [3 \Xi, 6 \Xi] \cap 2 \Z$ of size comparable to $\Xi$ and define 
\ALI{
\JJ &:= \Z \times (\Z/\Pi \Z)^3 \times \F \\
\F &:= \{ e_i \pm e_j : 1 \leq i < j \leq 3 \},
}
Each $V_J$, $J = (I, J_1, J_2, J_3, f)$, will be supported in a time interval of length $\sim \th$ around time $t(I)$, and initially at time $t(I)$ will be supported in a ball of size $\sim \Xi^{-1}$ around the point $x_0(J) := \Pi^{-1}(J_1, J_2, J_3) \in (\R / \Z)^3$.  The component $f \in \F$ specifies which of the $\# \F = 6$ directions in $\R^3$ in which $V_J^\ell$ approximately takes values.  

As in \cite[Section 12]{isett}, let $v_\ep = \eta_\ep \ast v$ be the coarse scale velocity field obtained by mollification in space at scale $\ep$.  Let $\Phi_s : \R \times \R \times \T^3 \to \R \times \T^3$ be the {\bf coarse scale flow} (the flow map of $v_\ep$)
\ali{ \label{eq:flowMap}
\begin{split}
\Phi_s(t,x) = (t+s, \Phi_s^i(t,x)), \quad 
\fr{d}{ds} \Phi_s^i(t,x) = v_\ep^i(\Phi_s(t,x)), \quad 
\Phi_0(t,x) = (t,x),
\end{split}
}
and let $\Ga_I : \R \times \T^3 \to \T^3$ be the {\bf back-to-labels map} associated to $v_\ep$ from the initial time $t(I)$
\ali{
\begin{split}
(\pr_t + v_\ep^i \nb_i) \Ga_I(t,x) &= 0 \\
\Ga_I(t(I), x) &= x.
\end{split}
}
We also define the coarse scale advective derivative $\Ddt := (\pr_t + v_\ep \cdot \nb)$.

To localize the waves $V_J$, we construct a smooth, quadratic partition of unity initiating from each time $t(I)$ that follows the flow of $v_\ep$ and has length scale $\sim \Xi^{-1}$.  The elements of this partition of unity are functions $\chi_{(I,[k])} : \R \times \T^3 \to \R$ that are indexed by $(I, [k]) \in \Z \times (\Z / \Pi \Z)^3$, and they satisfy
\ali{
\sum_{[k] \in (\Z / \Pi \Z)^3} \chi_{(I,[k])}^2(t,x) &= 1, \qquad \mbox{ for all } I \in \Z, (t,x) \in \R \times \T^3 \label{eq:quadPart} \\
\Ddt \chi_{(I,[k])}(t,x) &= 0, \qquad \mbox{ for all } (I,[k]) \in \Z \times (\Z / \Pi \Z)^3, (t,x) \in \R \times \T^3. \label{eq:transportPartUnity}
}
To construct the initial data for the partition of unity, choose a smooth $\bar{\chi} : \R^3 \to \R$ with support in $[-3/4, 3/4]^3$ such that $\sum_{m \in \Z^3} \bar{\chi}^2(h - m) = 1$ for all $h \in \R^3$, then periodize and rescale to define
\ali{
\chi_{(I,[k])}(t(I), x) &:= \sum_{m \in \Z^3} \bar{\chi}(\Pi x - [k] - \Pi m) . \label{eq:rescalingInit}
}
Observe that $\chi_{(I,[k])}(t(I), x)$ does not depend on how we represent the equivalence classes of $x \in (\R / \Z)^3$ or $[k] \in (\Z / \Pi \Z)^3$, and that \eqref{eq:quadPart} holds at time $t(I)$.  The same identity holds for all time $t \in \R$ by \eqref{eq:transportPartUnity} and uniqueness of solutions to the transport equation.  Observe also, since $3 \Xi \leq \Pi \leq 6 \Xi$, that the initial data for $\chi_{(I,[k])}(t(I), \cdot)$ is supported in a ball of radius $\Xi^{-1}$ around $\Pi^{-1} [k]$ in $(\R / \Z)^3$, and satisfies estimates of the form $\co{\nb_{\va}\chi_{(I,[k])}(t(I), \cdot)} \lsm_{|\va|} \Xi^{|\va|}$.

\subsection{Localizing the Convex Integration Construction} \label{sec:threadingMikadoFlows}
Unlike the scheme in \cite{isettOnsag}, our scheme will involve many Mikado flow based waves at any given time that are supported within overlapping regions.  In general, interference between overlapping Mikado flows would produce error terms that cannot be controlled for the iteration.  We avoid this interference by ``threading'' the Mikado flows together so that, at the initial time, the main terms of the waves $V_J$ will have disjoint support.  The support then remains disjoint as the Mikado flows are advected along the coarse scale flow.  

To accomplish this construction, let $f \in \F$ and let $[k] \in (\Z / 2 \Z)^3$.  Choose an $r_0 > 0$ and choose disjoint, periodic lines $\ell_{(f,[k])} = \{ p_{(f,[k])} + t f : t \in \R \}$ that are separated from each other by a distance greater than $6 r_0$ in the torus $(\R / \Z)^3$.  Choose smooth functions $\psi_{(f,[k])} : \T^3 \to \R$ of the form $\psi_{(f,[k])}(X) = g(\mbox{dist}(X, \ell_{(f,[k])} ))$, $\supp g(\cdot) \subseteq [r_0/2, r_0]$, such that
\ali{
\int_{\T^3} \psi_{(f,[k])}(X) dX = 0, \qquad  \int_{\T^3} \psi_{(f,[k])}^2(X) dX = 1. \label{eq:intZero}
}
With these choices, the functions $\psi_{(f,[k])}$ have disjoint support and have gradients orthogonal to $f$:
\ali{
\nb_\ell \psi_{(f,[k])}(X) f^\ell &= 0  \label{eq:orthog} \\
\supp \psi_{(f,[k])} \cap \supp \psi_{(\tilde{f},[\tilde{k}])} &= \emptyset \qquad \mbox{ if } f \neq \tilde{f} \mbox{ or } [k] \neq [\tilde{k}] \mbox{ in } (\Z / 2 \Z)^3. 
}
Conditions \eqref{eq:orthog} and \eqref{eq:intZero} imply that $\psi_{(f,[k])}(X) f^\ell$ is divergence free with mean zero, which implies that there is\footnote{We can take for instance $\Om_{(f,[k])}^{\a \b} = \nb^\a \De^{-1}[ \psi_{(f,[k])} f^\b ] - \nb^\b \De^{-1}[ \psi_{(f,[k])} f^\a ].$} a smooth tensor field $\Om_{(f,[k])}^{\a \b} : \T^3 \to \R^3 \otimes \R^3$ that is anti-symmetric in $\a \b$ and satisfies
\ali{
\nb_\a \Om_{(f,[k])}^{\a \b}(X) = \psi_{(f,[k])}(X) f^\b, \qquad \int_{\T^3} \Om_{(f,[k])}^{\a \b}(X) dX  = 0 \mbox{ for all } 1 \leq \a, \b \leq 3.
}
Since all components of the $\Om_{(f,[k])}^{\a \b}$ have mean $0$ on the torus, we can further construct tensor fields $\wtld{\Om}_{(f,[k])}^{\a \b \ga} : \T^3 \to \R^3 \otimes \R^3 \otimes \R^3$, also anti-symmetric in $\a \b$, such that
\ali{
\nb_\ga \wOm^{\a \b \ga}_{(f,[k])}(X) = \Om_{(f,[k])}^{\a \b}(X), \qquad  \int_{\T^3} \wOm_{(f,[k])}^{\a \b\ga}(X) dX  = 0 \mbox{ for all } 1 \leq \a, \b, \ga \leq 3.
}
For example, we can take $\wOm^{\a\b\ga}_{(f,[k])} := \nb^\ga \De^{-1} \Om_{(f,[k])}^{\a\b}$.  These second order potentials will be used to impose local conservation of angular momentum similar to the use of double-curl form waves in \cite{IOnonpd}.  

For $J = (I, J_1, J_2, J_3, f)$, let $[J] := [(J_1, J_2, J_3)]$.  We define the corrections $V_J^\ell$ to have the form
\ali{
\begin{split} \label{eq:formOfCorrection}
V_J^\ell &= \VR_J^\ell + \de V_J^\ell, \qquad
\VR_J^\ell = v_J^\ell \psi_J(t,x) \\
\psi_J(t,x) &:= \psi_{(f,[J])}(\la \Ga_I(t,x) ).
\end{split}
}
The amplitudes $v_J^\ell$ have the same form as in \cite[Section 13]{isettOnsag} except incorporating the partition of unity $\chi_J$.  In particular, they take values orthogonal to the gradient of the oscillatory functions $\psi_J$:
\ali{
v_J^\ell &= \chi_J [ e_I^{1/2}(t) \ga_{(I,f)}(t,x) (\nb \Ga_I^{-1})_a^\ell f^a ] \label{eq:amplitudeFormula} \\
\suppt e_I^{1/2}(t) &\subseteq [t(I) - \th, t(I) + \th] \label{eq:timeCutoffSupp}\\
\chi_J(t,x) &= \chi_{(I,[J_1, J_2, J_3])}(t,x), \quad J = (I, J_1, J_2, J_3, f) \\
v_J^\ell \nb_\ell \psi_J &= 0. \label{eq:orthogonality}
}
Note in particular that by construction the main terms of each wave have disjoint supports
\ali{
\supp \VR_J \cap \supp \VR_K &= \emptyset, \qquad \mbox{ if } J \neq K. \label{eq:disjointSupps}
}
Indeed, if $J = (J_0, J_1, J_2, J_3, f)$ and $K = (K_0, K_1, K_2, K_3, f')$ are not equal and $J_0 \neq K_0$, then $V_J^\ell$ and $V_K^\ell$ live on different time intervals.  If $J_0 = K_0 = I$, one has either $f \neq f'$ or $(J_1, J_2, J_3) \neq (K_1, K_2, K_3)$ mod $2$, either case implying $\supp \psi_J \cap \supp \psi_K = \emptyset$, or $f = f'$ and $(J_1, J_2, J_3) = (K_1, K_2, K_3)$ mod $2$.  In the last case, one has $\supp \chi_J \cap \supp \chi_K = \emptyset$ unless $J = K$.



The correction $V_J^\ell$ is made to be divergence free and to have the form \eqref{eq:formOfCorrection} by making $V_J^\ell$ the divergence of an antisymmetric tensor built from the Lie transport of the potentials $\wOm_{(f,[k])}^{\a \b\ga}$ above:
\ali{
V_J^\ell &= \la^{-2} \nb_a\nb_c[ \chi_J (\nb \Ga_I^{-1})^a_\a (\nb \Ga_I^{-1})_\b^\ell(\nb \Ga_I^{-1})_\ga^c  e^{1/2}_I(t) \ga_{(I,f)} \wOm_J^{\a \b\ga} ]  \label{eq:doubleDivForm}\\
\de V_J^\ell &= \de v_{J, \a \b}^\ell \Om_J^{\a\b} + \de v_{J,\a\b\ga}^\ell \wOm_J^{\a \b \ga} \\
\begin{split}\label{eq:OmJabForm} 
\Om_J^{\a\b}(t,x) &:= \Om_{(f,[J_1, J_2, J_3])}^{\a \b}(\la \Ga_I) \\
 \wOm_J^{\a \b \ga}(t,x) &:= \wOm_{(f,[J_1, J_2, J_3])}^{\a \b\ga}(\la \Ga_I) 
\end{split} \\
\begin{split}
\de v_{J, \a \b}^\ell &:= \la^{-1} \nb_a[ \chi_J (\nb \Ga_I^{-1})^a_\a (\nb \Ga_I^{-1})_\b^\ell e^{1/2}_I(t) \ga_{(I,f)} ] 
\label{eq:devJabForm}\\
\de v_{J, \a \b \ga}^\ell &:= \la^{-2} \nb_a \nb_c[ \chi_J (\nb \Ga_I^{-1})^a_\a (\nb \Ga_I^{-1})_\b^\ell(\nb \Ga_I^{-1})_\ga^c  e^{1/2}_I(t) \ga_{(I,f)} ].
\end{split}
}
Note that the main term $\VR_J^\ell$ in \eqref{eq:formOfCorrection}-\eqref{eq:amplitudeFormula} appears when the derivatives $\nb_a \nb_c$ both fall on $\wOm_J^{\a \b \ga}$.  Since $V_J^\ell$ has the form $V_J^\ell = \nb_a W_J^{a\ell}$, where $W_J^{a\ell}$ is anti-symmetric in $a\ell$, we have that $V_J^\ell$ is divergence free.

The amplitudes constructed here are related to those constructed in \cite[Section 13]{isettOnsag} (which are indexed by $(I, f) \in \Z \times \F$ and do not involve spatial cutoffs) by the formula
\ali{
v_J^\ell &= \chi_J v_{(I, f)}^\ell, \qquad J = (I, J_1, J_2, J_3, f). \label{eq:relatedFormula} 
}
This comparison allows us to see that the parameter $\ep = \ep_v$ in the mollification of $v \mapsto v_\ep$ can be chosen to have the same value $\ep_v = c_v N^{-1/2} \Xi^{-1}$ as in \cite[Section 16]{isettOnsag}, which is based on the requirement 
\ali{
\co{ v - v_\ep } \max_J \co{|v_J| |\psi_J| }
&\leq (\log \hxi)^{1/2} \fr{e_v^{1/2} e_R^{1/2}}{500 N}. \label{eq:requireEpv}
}

Since we have chosen the same parameter in the mollification $v \mapsto \ep_v$ as that chosen in \cite{isettOnsag}, we obtain the same estimates for $v_\ep$
\ali{
\co{\nb_{\va} v_\ep} &\lsm_{|\va|} N^{(|\va| - 2)_+/2} \Xi^{|\va|} e_v^{1/2} \quad \mbox{ if } |\va| \geq 1, \label{eq:vepBds}
}
where the implicit constant is equal to $1$ for $|\va| = 1$.  From this fact we will see in the following Section~\ref{sec:componentBounds} that all the remaining estimates for the components of the construction coincide with those in the proof of \cite[Lemma 3.3]{isettOnsag}.

\section{Estimates for Components of the Construction} \label{sec:componentBounds}
Here we summarize the estimates for the components of the construction, which coincide with those of \cite{isettOnsag}.  The following elementary Lemma will be convenient:
\begin{lem} \label{lem:elementaryLem} For $u \geq 0$, integer $M \geq 0$ and for $g : \T^3 \to \R$ define (for $N \geq 1, \Xi > 0$)
\ali{
H_{M,u}[g] &:= \max_{0 \leq |\va| \leq M} \fr{\co{\nb_{\va} g}}{N^{(|\va| - u)_+/2} \Xi^{|\va|} } \label{eq:weightedNormDef}
}
Then for $\la \geq N^{1/2} \Xi$, we have for any first order partial derivative $\nb_a$
\ali{
\begin{split}
H_{M,u}[\la^{-1} \nb_a g] \leq H_{M+1, u}[g], \\
 H_{M,u}[\Xi^{-1} \nb_a g] \leq H_{M+1,u+1}[g]
\end{split}
}
We also have the triangle inequality $H_{M,u}[g_{(1)} + g_{(2)}] \leq H_{M,u}[g_{(1)}] + H_{M,u}[g_{(2)}]$ and product estimate
\ali{
H_{M,u}[g_{(1)}g_{(2)}] \lsm_M H_{M,u}[g_{(1)}] H_{M,u}[g_{(2)}]. \label{eq:productBound}
}
\end{lem}
All the properties follow quickly from the definition \eqref{eq:weightedNormDef}.  Inequality \eqref{eq:productBound} follows from the expansion $\nb_{\va}(g_{(1)} g_{(2)} ) = \sum_{|\va_1| + |\va_2| = |\va|} c_{\va_1, \va_2} \nb_{\va_1} g_{(1)} \nb_{\va_2} g_{(2)}$, the bound $\co{\nb_{\va_i} g_{(i)}} \leq N^{(|\va_i|-u)_+/2} \Xi H_{M,u}[g_{(i)}]$ and the inequality $(|\va_1| - u)_+ + (|\va_2| - u)_+ \leq (|\va|-u)_+$.

The estimates for the construction may now be summarized as follows.  Here we use the fact that the frequency $\la := B_\la N \Xi$ is larger than $N^{1/2} \Xi$ to conclude that the lower order terms $\de v_{J, \a \b \ga}^\ell$ obey the same bounds as the $\de v_{J, \a \b}^\ell$.  
\begin{prop}  The following bounds hold with constants depending only on $|\va|$
\ali{
\co{\nb_{\va} \ga_{(I,f)} } + \co{\nb_{\va} (\nb \Ga_I^{-1}) } &\lsm N^{(|\va| - 1)_+/2} \Xi^{|\va|} \label{eq:gaIbounds} \\
\co{\nb_{\va} \Ddt \ga_{(I,f)} } +  \co{\nb_{\va} \Ddt (\nb \Ga_I^{-1}) }&\lsm \plhxi^2 \Xi e_v^{1/2} N^{(|\va| - 1)_+/2} \Xi^{|\va|} \label{eq:DdtbackmapBds} \\
\sup_{t \in \R} \left(e_I^{1/2}(t) + \th \left| \pr_t e_I^{1/2}(t) \right| \right) &\lsm \plhxi^{1/2} e_R^{1/2} \label{eq:timeCutoffBds} \\
\co{ \nb_{\va} \chi_J} &\lsm N^{(|\va|-1)_+/2} \Xi^{|\va|} \label{eq:chiJ} \\
\co{ \nb_{\va} v_J^\ell } &\lsm \plhxi^{1/2} N^{(|\va| - 1)_+/2} \Xi^{|\va|} e_R^{1/2} \label{eq:amplitudeBounds} \\
\co{ \nb_{\va} \Ddt v_J^\ell } &\lsm \plhxi^{5/2} N^{(|\va| - 1)_+/2} \Xi^{|\va|} e_R^{1/2} \\
\co{ \nb_{\va} \de v_{J,\a\b}^\ell } + \co{ \nb_{\va} \de v_{J,\a\b\ga}^\ell } &\lsm \la^{-1} \plhxi^{1/2} N^{|\va|/2} \Xi^{1 + |\va|} e_R^{1/2} \label{eq:lowerOrderAmpbds} \\
\co{ \nb_{\va} \Ddt \de v_{J,\a\b}^\ell } + \co{ \nb_{\va} \Ddt \de v_{J,\a\b\ga}^\ell } &\lsm \la^{-1} \plhxi^{5/2} N^{|\va|/2} \Xi^{|\va| + 2} e_v^{1/2} e_R^{1/2} \label{eq:lowOrderDtAmpbds}
}
\end{prop}
\begin{proof}  Inequalities \eqref{eq:gaIbounds}-\eqref{eq:timeCutoffBds} follow from the bounds in \cite[Section 17.1]{isettOnsag}.  
Inequality \eqref{eq:chiJ} for $|\va| = 0$ follows from the maximum principle for $\Ddt \chi_J = 0$.  To obtain \eqref{eq:chiJ}, we apply \cite[Proposition 17.4]{isett} in the case of order $L=2$ frequency-energy levels to obtain 
\ali{
E_M[\chi_J](\Phi_s(t,x)) &\leq e^{C_M \Xi e_v^{1/2} |s|} E_M[\chi_J](t(I), x) \label{eq:dimEnGrowthbd} \\
E_M[\chi_J](t,x) &:= \sum_{0 \leq |\va| \leq M} \Xi^{-2|\va|} N^{-(|\va| - 1)_+} |\nb_{\va} \nb \chi_J(t,x)|^2,
}
and we use the fact that, by the construction in \eqref{eq:rescalingInit},
\ali{
E_M[\chi_J](t(I), x) &\lsm_M \sum_{0 \leq |\va| \leq M} \Xi^{-2|\va|} N^{-(|\va| - 1)_+} ( \Xi^{|\va| + 1} )^2 \lsm_M \Xi^2.
}
We have  $\Xi e_v^{1/2} |s| \leq \Xi e_v^{1/2} \th \leq 1$ on the support of the time cutoff $e_I^{1/2}$ from \eqref{eq:timeCutoffSupp}, so \eqref{eq:dimEnGrowthbd} yields $\co{ \nb_{\va} \chi_J} \lsm N^{(|\va|-2)_+/2} \Xi^{|\va|}$, which implies \eqref{eq:chiJ}.

The proofs of estimates \eqref{eq:amplitudeBounds}-\eqref{eq:lowOrderDtAmpbds} for $v_J^\ell$ and for $\de v_{J, \a \b \ga}^\ell$ are exactly as in \cite[Section 17.1]{isettOnsag} with the addition of the cutoff function $\chi_J$.  For instance, note that $\chi_J (\nb \Ga_I^{-1})_\a^a$ and $\Ddt [ \chi_J (\nb \Ga_I^{-1})_\a^a ] = \chi_J \Ddt (\nb \Ga_I^{-1})_\a^a $ obey the same bounds as $(\nb \Ga_I^{-1})_\a^a$ and $\Ddt (\nb \Ga_I^{-1})_\a^a$ respectively up to constants, so we may absorb the cutoff $\chi_J$ into the first factor of $(\nb \Ga^{-1})$ in estimating formulas \eqref{eq:amplitudeFormula} and \eqref{eq:devJabForm} while repeating the proofs in \cite[Section 17.1]{isettOnsag}.

It remains to check \eqref{eq:lowerOrderAmpbds}-\eqref{eq:lowOrderDtAmpbds} for the lower order term $\de v_{J,\a\b\ga}^\ell$.  Applying Lemma~\ref{lem:elementaryLem}, we obtain
\ali{
\begin{split}
\notag
\la \Xi^{-1} \de v_{J,\a\b\ga}^\ell &=  \Xi^{-1} \la^{-1} \nb_a \nb_c[ \chi_J (\nb \Ga_I^{-1})^a_\a (\nb \Ga_I^{-1})_\b^\ell(\nb \Ga_I^{-1})_\ga^c  e^{1/2}_I(t) \ga_{(I,f)} ] \\
\la \Xi^{-1} H_{M,0}[\de v_{J,\a\b\ga}^\ell] &\lsm_M H_{M+1,1}[\la^{-1} \nb_c[ \chi_J (\nb \Ga_I^{-1})^a_\a (\nb \Ga_I^{-1})_\b^\ell(\nb \Ga_I^{-1})_\ga^c  e^{1/2}_I(t) \ga_{(I,f)} ] \\
&\lsm_M H_{M+2,1}[\chi_J (\nb \Ga_I^{-1})^a_\a (\nb \Ga_I^{-1})_\b^\ell(\nb \Ga_I^{-1})_\ga^c  e^{1/2}_I(t) \ga_{(I,f)}] \end{split} \\
&\lsm_M e_I^{1/2}(t) H_{M+2,1}[\chi_J] H_{M+2,1}[(\nb \Ga_I^{-1})]^3 H_{M+2,1}[\ga_{(I,f)}] \label{eq:almostLast}\\
H_{M,0}[\de v_{J,\a\b\ga}^\ell] &\lsm_M \la^{-1} \plhxi^{1/2} \Xi e_R^{1/2}. \label{eq:bdWeWanted}
}
Here every term in \eqref{eq:almostLast} is bounded by $\lsm_M 1$ except $e_I^{1/2}(t)$.  Note that \eqref{eq:bdWeWanted} is equivalent to \eqref{eq:lowerOrderAmpbds}.

To prove \eqref{eq:lowOrderDtAmpbds}, we proceed similarly by commuting in the advective derivative weighted by the parameter $\th \sim \plhxi^{-2} \Xi^{-1} e_v^{-1/2}$:
\ali{
(\la \Xi^{-1} \th) \Ddt \de v_{J,\a\b\ga}^\ell &= \Xi^{-1} \la^{-1} \nb_a \nb_c[\th \Ddt [\chi_J (\nb \Ga_I^{-1})^a_\a (\nb \Ga_I^{-1})_\b^\ell(\nb \Ga_I^{-1})_\ga^c  e^{1/2}_I(t) \ga_{(I,f)} ] ] \label{eq:commutedTermdlt} \\
&- \th (\nb_a v_\ep^i) \Xi^{-1} \la^{-1} \nb_i \nb_c[\chi_J (\nb \Ga_I^{-1})^a_\a (\nb \Ga_I^{-1})_\b^\ell(\nb \Ga_I^{-1})_\ga^c  e^{1/2}_I(t) \ga_{(I,f)} ] \label{eq:commutermsFirst} \\
&- \Xi^{-1}\nb_a[ \nb_c v_\ep^i \la^{-1}\nb_i[\chi_J (\nb \Ga_I^{-1})^a_\a (\nb \Ga_I^{-1})_\b^\ell(\nb \Ga_I^{-1})_\ga^c  e^{1/2}_I(t) \ga_{(I,f)} ] ]. \label{eq:commutermsSecond}
}
The terms \eqref{eq:commutermsFirst} and \eqref{eq:commutermsSecond} may be estimated using Lemma~\ref{lem:elementaryLem} as in the proof of \eqref{eq:almostLast}-\eqref{eq:bdWeWanted} to obtain
\ALI{
H_{M,0}[\eqref{eq:commutermsFirst}] + H_{M,0}[\eqref{eq:commutermsSecond}] &\lsm_M e_I^{1/2}(t) H_{M+1,1}[ \th \nb v_\ep ] H_{M+2,1}[\chi_J] H_{M+2,1}[(\nb \Ga_I^{-1})]^3 H_{M+2,1}[\ga_{(I,f)}] \\
\stackrel{\eqref{eq:vepBds}-\eqref{eq:chiJ} }{\lsm_M}& e_I^{1/2}(t) \lsm \plhxi^{1/2} e_R^{1/2}.
}
For \eqref{eq:commutedTermdlt}, apply the product rule for $\th \Ddt$ and apply Lemma~\ref{lem:elementaryLem} repeatedly to obtain 
\ali{
\begin{split}
H_{M,0}[\eqref{eq:commutedTermdlt}] \lsm_M&  \left(e_I^{1/2}(t) + \th \left|\pr_t e_I^{1/2}(t) \right|\right) H_{M+2,1}[\chi_J] \cdot \left(H_{M+2,1}[(\nb \Ga_I^{-1})] + \th H_{M+2,1}[\Ddt(\nb \Ga_I^{-1})]\right)^3 \\
&\cdot (H_{M+2,1}[\ga_{(I,f)}] + \th H_{M+2,1}[\Ddt \ga_{(I,f)}]). \notag
\end{split}
}
Since $\th H_{M+2,1}[\Ddt \ga_{(I,f)}]$ and $\th H_{M+2,1}[\Ddt(\nb \Ga_I^{-1})]$ are bounded by $\lsm_M 1$ from \eqref{eq:gaIbounds}-\eqref{eq:DdtbackmapBds}, we have
\ALI{
H_{M,0}[\de v_{J,\a\b\ga}] &\leq \th^{-1} \la^{-1} \Xi (H_{M,0}[\eqref{eq:commutedTermdlt}] + H_{M,0}[\eqref{eq:commutermsFirst}] + H_{M,0}[\eqref{eq:commutermsSecond}] ) \\
&\lsm_M \th^{-1} \la^{-1} \Xi \left(e_I^{1/2}(t) + \th \left|\pr_t e_I^{1/2}(t) \right|\right) \lsm \th^{-1} \la^{-1} \Xi \plhxi^{1/2} e_R^{1/2}.
}
This bound is equivalent to the desired bound \eqref{eq:lowOrderDtAmpbds} for $\de v_{J,\a\b\ga}$.
\end{proof}
As \eqref{eq:amplitudeBounds}-\eqref{eq:lowOrderDtAmpbds} are the same bounds for the components of the correction as those proven for $v_{(I,f)}^\ell$ and $\de v_{(I,f),\a\b}^\ell$ in \cite[Section 17]{isettOnsag}, we have the following bounds from \cite[Proposition 17.3]{isettOnsag}.
\begin{prop}[Correction estimates]  For $0 \leq |\va| \leq 3$, we have
\ali{
\sup_J \co{\nb_{\va} \VR_J} &\lsm (B_\la N \Xi)^{|\va|} \plhxi^{1/2} e_R^{1/2} \\
\sup_J \co{\nb_{\va} \de V_J} &\lsm (B_\la N \Xi)^{|\va| - 1} \Xi \plhxi^{1/2} e_R^{1/2} \\
\co{ V } &\lsm (B_\la N \Xi)^{|\va|} \plhxi^{1/2} e_R^{1/2}  \label{eq:c0CorrectBound} \\
\suppt V \subseteq &\bigcup_I \suppt e_I \subseteq \bigcup_I [t(I) - \th, t(I) + \th].
}
\end{prop}
\noindent For the estimate \eqref{eq:c0CorrectBound}, we use that at most a bounded number (say $2^3$) distinct $V_J^\ell$ are supported at any given point $(t,x)$.  This detail will be explained following equation \eqref{eq:finallyBddRHT} below.  We now consider the error terms and their estimates.

\section{The Error Terms} \label{sec:errorTerms}
Given the Euler-Reynolds flow $(v,p,R)$, the new velocity field $v_1^\ell = v^\ell + V^\ell$, $V^\ell = \sum_J V_J^\ell = \sum_J \VR_J^\ell + \de V_J^\ell$ and pressure $p_1 = p + P$ will solve the Euler-Reynolds equations when coupled to a new Reynolds stress tensor $R_1^{j\ell}$.  The new stress $R_1^{j\ell}$ will be composed of terms that solve
\ali{
R_1^{j\ell} &= R_M^{j\ell} + R_T^{j\ell} + R_S^{j\ell} + R_H^{j\ell} \label{eq:R1defterms} \\
R_M^{j\ell} &= (v^j - v_\ep^j) V^\ell + V^j (v^\ell - v_\ep^\ell) + (R^{j\ell} - R_\ep^{j\ell}) \\
\nb_j R_T^{j\ell} &= \pr_t V^\ell + \nb_j ( v_\ep^j V^\ell + V^j v_\ep^\ell) \label{eq:transTerm} \\
R_S^{j\ell} &= \sum_{J,K \in \JJ} \de V_J^j \VR_K^\ell + \VR_J^j \de V_K^\ell + \de V_J^j \de V_K^\ell \\
\nb_j R_H^{j\ell} &= \nb_j \left[ \sum_{J\in\JJ} \VR_J^j \VR_J^\ell + P \de^{j\ell} + R_\ep^{j\ell} \right]. \label{eq:highTerm}
}
In writing \eqref{eq:highTerm}, we have made the crucial observation that all of the off-diagonal terms in the summation $\sum_{J,K \in \JJ} \VR_J^j \VR_K^\ell$ vanish due to the disjointness of support stated in \eqref{eq:disjointSupps}.  

Our construction has been designed in such a way that
\ali{
\sum_{J\in\JJ} v_J^j v_J^\ell + P \de^{j\ell} + R_\ep^{j\ell} &= 0. \label{eq:cancelStress}
}
From \eqref{eq:formOfCorrection} and \eqref{eq:cancelStress}, equation \eqref{eq:highTerm} reduces to
\ali{
\nb_j R_H^{j\ell} &= \nb_j \left[ \sum_{J\in\JJ} v_J^j v_J^\ell (\psi_J^2 - 1) \right]. \label{eq:reducedRHform}
}
To verify \eqref{eq:cancelStress}, note that for each $I \in \Z$ and $\JJ(I) := \{ I \} \times (\Z / \Pi \Z)^3 \times \F$, we have from \eqref{eq:quadPart}, \eqref{eq:relatedFormula} that
\ali{
\sum_{J\in\JJ(I)} v_J^j v_J^\ell &= \sum_{[k] \in (\Z / \Pi \Z)^3} \sum_{f \in \F} \chi_{(I, [k])}^2 v_{(I,f)}^j v_{(I,f)}^\ell = \sum_{f \in \F} v_{(I,f)}^j v_{(I,f)}^\ell, \label{eq:relateLowFreqPart}
}
where $v_{(I,f)}$ are the amplitudes from the construction in \cite{isettOnsag}.  The equality
\ali{
\sum_{I \in \Z} \sum_{f \in \F} v_{(I,f)}^j v_{(I,f)}^\ell + P \de^{j\ell} + R_\ep^{j\ell} &= 0
}
proved in \cite[Sections 14-15]{isettOnsag} now implies the equality \eqref{eq:cancelStress} in the present construction using \eqref{eq:relateLowFreqPart}.

It now remains to show that, when $R_T^{j\ell}$ and $R_H^{j\ell}$ are chosen appropriately, the tensor $R_1^{j\ell}$ defined by \eqref{eq:R1defterms} satisfies the bounds required by Lemma~\ref{lem:convexInt}.

\section{Solving the Symmetric Divergence Equation} \label{sec:solveSymmdiv}
To estimate the error tensor $R_1$ defined in \eqref{eq:R1defterms}, the only terms that require a different treatment from \cite{isettOnsag} are the terms $R_T$ and $R_H$.  Namely, since our choice of $v_\ep$ and $R_\ep$ and our estimates for $\VR_J$ and $\de V_J$ also coincide with those of \cite{isettOnsag}, Proposition 17.4 from \cite{isettOnsag} shows that
\ali{
\co{R_M} + \co{R_S} &\leq \plhxi \fr{e_v^{1/2} e_R^{1/2}}{10 N} \label{eq:c0bdRMS}\\
\co{\nb_{\va} R_M} + \co{\nb_{\va} R_S} &\lsm (B_\la N \Xi)^{|\va|} (\lhxi) \fr{e_v^{1/2} e_R^{1/2}}{N}, \qquad 1 \leq |\va| \leq 3 \\
\suppt R_M \cup \suppt R_S &\subseteq \bigcup_I[t(I) - \th, t(I) + \th]
}
provided we choose the constant $B_\la$ in the definition of $\la = B_\la N \Xi$ larger than a certain, absolute constant $\overline{B}_\la$.

The tensors $R_T$ and $R_H$ are defined as summations of the form
\ali{
R_T^{j\ell} = \sum_{J \in \JJ} R_{T, J}^{j\ell}, \qquad R_H^{j\ell} = \sum_{J \in \JJ} R_{H, J}^{j\ell}, \label{eq:RTHdef}
}
where each term is symmetric and is localized both in space and in time around the support of $V_J^\ell$.

We expand the terms \eqref{eq:transTerm} and \eqref{eq:reducedRHform} (using the orthogonality $v_J^j \nb_j \psi_J = 0$ stated in \eqref{eq:orthogonality} in the case of $R_H$, and using $\nb_j v_\ep^j = \nb_j V_J^j = 0$ in the case of $R_T$) to obtain the equations
\ali{
\nb_j R_{T,J}^{j\ell} &= \pr_t V_J^\ell + \nb_j(v_\ep^j V_J^\ell + V_J^j v_\ep^\ell) \label{eq:RTJdef} \\
\nb_j R_{T,J}^{j\ell} &= u_{TJ}^\ell \psi_J + u_{TJ, \a \b}^\ell \Om^{\a \b}_J + u_{TJ, \a\b\ga}^\ell \wOm^{\a\b\ga}_J \label{eq:transportTermFull} \\
\nb_j R_{H,J}^{j\ell} &= u_{HJ}^\ell (\psi_J^2 - 1) \label{eq:highTermFull} \\
u_{HJ}^\ell &= \nb_j[ v_J^j v_J^\ell] \notag \\
u_{TJ}^\ell &:= \Ddt v_J^\ell + v_J^j \nb_j v_\ep^\ell, \quad \notag \\
u_{TJ, \a \b}^\ell &:= \Ddt \de v_{J,\a\b}^\ell + \de v_{J,\a\b}^j \nb_j v_\ep^\ell, \quad \notag \\
u_{TJ, \a \b\ga}^\ell &:= \Ddt \de v_{J,\a\b\ga}^\ell + \de v_{J,\a\b\ga}^j \nb_j v_\ep^\ell. \notag 
}
By the construction in Section~\ref{sec:threadingMikadoFlows}, each of the functions $\psi_J$, $(\psi_J^2 - 1)$, $\Om^{\a\b}_J$ and $\wOm^{\a\b\ga}_J$ have the form $\om(\la \Ga_I(t,x))$, where $\om : \T^3 \to \R$ belongs to a finite set of smooth functions that mean zero on $\T^3$.  We may therefore apply the following Proposition, which is similar to \cite[Proposition 17.6]{isettOnsag} and is proven in Section~\ref{sec:paramExpand} below using the same parametrix expansion technique.
\begin{prop}[Nonstationary Phase]  \label{prop:parametric} If $U^\ell : \R \times \T^3 \to \R^3$ is a smooth vector field of the form
$U^\ell = u^\ell \om(\la \Ga_I)$, 
where $\om : \T^3 \to \R$ is a smooth function of mean zero, then for any $D \geq 1$ there exist a smooth, symmetric tensor field $Q^{j\ell}_{(D)} : \R \times \T^3 \to \R^3 \otimes \R^3$ and a vector field $U_{(D)}^{\ell}$ satisfying
\ali{
U^\ell &= \nb_j Q_{(D)}^{j\ell} + U_{(D)}^\ell \\
\sup_{0 \leq |\va| \leq 3} \la^{-|\va|} \co{\nb_{\va} Q_{(D)}^{j\ell} } &\lsm \la^{-1} \sup_{0 \leq |\va| \leq D + 3} \fr{\co{\nb_{\va} u^\ell} }{N^{|\va|/2} \Xi^{|\va|}} \\
\sup_{0 \leq |\va| \leq 3} \la^{-|\va|} \co{\nb_{\va} U^\ell } &\lsm B_\la^{-1} N^{-D/2} \sup_{0 \leq |\va| \leq D + 3} \fr{\co{\nb_{\va} u^\ell} }{N^{|\va|/2} \Xi^{|\va|}} \\
\supp Q_{(D)}^{j\ell} &\cup \supp U_{(D)}^\ell \subseteq \supp U^\ell,
}
where the implicit constant depends only on $\om$ and $D$.
\end{prop}

We apply Proposition~\ref{prop:parametric} to each of the terms in \eqref{eq:transportTermFull} and \eqref{eq:highTermFull} and use the estimates
\ALI{
H_{D+3,0}[u_{TJ}^\ell] + H_{D+3,0}[u_{TJ,\a\b}^\ell] + H_{D+3,0}[&u_{TJ,\a\b\ga}^\ell] + H_{D+3,0}[u_{HJ}^\ell] \lsm \plhxi^{5/2} \Xi e_v^{1/2} e_R^{1/2} \\
H_{D+3,0}[u] &:= \sup_{0 \leq |\va| \leq D + 3} \fr{\co{\nb_{\va} u^\ell} }{N^{|\va|/2} \Xi^{|\va|}},
}
which follow from \eqref{eq:amplitudeBounds}-\eqref{eq:lowOrderDtAmpbds} and Lemma~\ref{lem:elementaryLem} (and are saturated only by $u_{TJ}^\ell$), to obtain decompositions
\ali{
\eqref{eq:transportTermFull} &= \nb_j Q_{TJ,(D)}^{j\ell} + U_{TJ,(D)}^\ell, \qquad 
\eqref{eq:highTermFull} = \nb_j Q_{HJ,(D)}^{j\ell} + U_{HJ,(D)}^\ell \label{eq:defineRemainders}
}
where the symmetric tensors $Q_{TJ,(D)}$ and $Q_{HJ,(D)}$ and remainder terms $U_{TJ, (D)}$ and $U_{HJ, (D)}$ satisfy
\ali{
\sup_{0 \leq |\va| \leq 3} \la^{-|\va|} (\co{\nb_{\va} Q_{TJ, (D)}^{j\ell} } + \co{\nb_{\va} Q_{HJ, (D)}^{j\ell} } ) &\lsm_D \la^{-1} \plhxi^{5/2} \Xi e_v^{1/2} e_R^{1/2} \\
&\lsm_D B_\la^{-1} \plhxi^{5/2} \fr{e_v^{1/2} e_R^{1/2}}{N} \label{eqQHJDbds}\\
\sup_{0 \leq |\va| \leq 3} \la^{-|\va|} (\co{\nb_{\va} U_{TJ, (D)}^{j\ell} } + \co{\nb_{\va} U_{HJ, (D)}^{j\ell} } ) &\lsm_D B_\la^{-1} N^{-D/2} \plhxi^{5/2} \Xi e_v^{1/2} e_R^{1/2} \label{sec:UTHJDbd} \\
\supp U_{TJ, (D)} \cup \supp U_{HJ, (D)} \cup \supp Q_{TJ,(D)} &\cup \supp Q_{HJ,(D)} \subseteq \supp \chi_J \cdot e_I^{1/2}(t). \label{eq:suppparamCt}
}
To complete the construction of $R_{T,J}^{j\ell}$ and $R_{HJ}^{j\ell}$ to \eqref{eq:transportTermFull}-\eqref{eq:highTermFull}, we construct solutions to the equations 
\ali{
\nb_j R_{TJ,(D)}^{j\ell} = U_{TJ,(D)}^\ell, \qquad \nb_j R_{HJ,(D)}^{j\ell} = U_{HJ,(D)}^\ell \label{eq:solveSymmDivRem}
}
that are localized around space-time cylinders containing the supports of $v_J$ by using the inverses for the symmetric divergence equation that were constructed in \cite{IOnonpd}.  We first recall the notions of Lagrangian and Eulerian cylinders from \cite{IOnonpd}.
\begin{defn} Let $\Phi_s$ be the flow map associated to $v_\ep$ as defined in \eqref{eq:flowMap}.  Given a point in space-time $(t_0, x_0) \in \R \times \T^3$ and positive numbers $\tau, \rho > 0$, we define the {\it $v_\ep$-adapted Eulerian cylinder} $\ECyl(\tau, \rho; t_0, x_0)$ with duration $2 \tau$ and base radius $\rho$ as well as the {\it $v_\ep$-adapted Lagrangian cylinder} $\LCyl(\tau, \rho; t_0, x_0)$ with duration $2 \tau$ and base radius $\rho$ to be
\ali{
\ECyl(\tau, \rho; t_0, x_0) &:= \{ \Phi_s(t_0, x_0) + (0, h) : 0 \leq |s| \leq \tau, 0 \leq |h| \leq \rho \} \\
\LCyl(\tau, \rho; t_0, x_0) &:= \{ \Phi_s(t_0, x_0 + h) : 0 \leq |s| \leq \tau, 0 \leq |h| \leq \rho \}
}
\end{defn}
The two notions are related (see \cite[Lemma 5.2]{IOnonpd}) by
\ali{
(t', x') \in \ECyl(\tau, \rho; t_0, x_0) \iff (t,x) \in \LCyl_v(\tau, \rho; t', x') \label{eq:duality}\\
\LCyl(\tau, e^{-\tau \co{\nb v_\ep}} \rho; t_0, x_0) \subseteq \ECyl(\tau, \rho; t_0, x_0) \subseteq \LCyl(\tau, e^{\tau \co{\nb v_\ep}} \rho; t_0, x_0). \label{eq:containCylsequiv}
}
It follows that the amplitudes constructed in Section~\ref{sec:threadingMikadoFlows} are supported in an Eulerian cylinder
\ali{
\supp \chi_J \cdot e_I^{1/2}(t) \subseteq \LCyl(\th, \Pi^{-1}; &t(I), x_0(J) ) \subseteq \ECyl(\th, e^{\th \co{\nb v_\ep} } \Pi^{-1}; t(I), x_0(J) ) \notag \\
\supp \chi_J \cdot e_I^{1/2}(t) &\subseteq \ECyl(\th, \Xi^{-1}; t(I), x_0(J) ), \label{eq:containsuppvJ}
}
and the remainder terms $U_{TJ, (D)}^\ell$ and $U_{HJ, (D)}^\ell$ are supported in the same Eulerian cylinder by \eqref{eq:suppparamCt}.

Before we can obtain symmetric tensors that solve the equations in \eqref{eq:solveSymmDivRem}, we must check that the necessary orthogonality conditions
\ali{
\int_{\R^3} U^\ell(t,x) dx = 0, \quad \int_{\R^3} ( x^j U^\ell - x^\ell U^j)(t,x) dx = 0, \qquad 1 \leq j, \ell \leq 3 \label{eq:necOrthogCondition}
}
are satisfied, where $U^\ell$ is the (nonperiodic restriction of) $U_{TJ,(D)}^\ell$ or $U_{HJ,(D)}^\ell$.  To check condition~\eqref{eq:necOrthogCondition}, note that $U_{HJ,(D)}^\ell$ is by construction in \eqref{eq:reducedRHform},\eqref{eq:defineRemainders} the divergence of a smooth {\it symmetric} tensor with compact support, and that $U_{TJ,(D)}^\ell$ has the form $\nb_a \nb_c [T_J^{a c \ell} ] + \nb_j U_J^{j\ell}$ (using \eqref{eq:doubleDivForm},\eqref{eq:RTJdef},\eqref{eq:defineRemainders}), where $U_J^{j\ell}$ is symmetric and both $T_J^{a c \ell}$ and $U_J^{j\ell}$ have compact support in the cylinder \eqref{eq:containsuppvJ}.  Integrating by parts, one obtains the conditions \eqref{eq:necOrthogCondition} for the nonperiodic restrictions of both $U_{TJ,(D)}^\ell$ and $U_{HJ,(D)}^\ell$.

We can now apply the operators in \cite[Section 10]{IOnonpd} (in particular Lemmas 10.3 and 10.4 with $\bar{\rho} = \Xi^{-1}$) to obtain symmetric tensors solving \eqref{eq:solveSymmDivRem} such that
\ali{
\supp R_{TJ,(D)}^{j\ell} \cup \supp R_{HJ,(D)}^{j\ell} &\subseteq \ECyl(\th, \Xi^{-1}; t(I), x_0(J) ) \label{eq:suppContain} \\
\co{R_{TJ,(D)}^{j\ell}} + \co{R_{HJ,(D)}^{j\ell}} &\lsm \Xi^{-1} (\co{U_{TJ,(D)}^\ell} + \co{U_{HJ,(D)}^\ell} ) \notag \\
\co{R_{TJ,(D)}^{j\ell}} + \co{R_{HJ,(D)}^{j\ell}} &\stackrel{\eqref{sec:UTHJDbd}}{\lsm} B_\la^{-1} N^{-D/2}\plhxi^{5/2} e_v^{1/2} e_R^{1/2}  \label{eq:coBdRemainderSol} \\
 \co{\nb_{\va} R_{TJ,(D)}^{j\ell}} + \co{\nb_{\va} R_{HJ,(D)}^{j\ell}} &\lsm_{|\va|} \Xi^{-1} \sum_{|\vcb| \leq |\va|} \Xi^{|\va| - |\vcb|} (\co{\nb_{\vcb} U_{TJ,(D)}^\ell} + \co{\nb_{\vcb} U_{HJ,(D)}^\ell} ). \label{eq:derivRmdrBdSoln}
}
We now set $D = 2$ and define $R_{TJ}^{j\ell} = Q_{TJ,(D)}^{j\ell} + R_{TJ,(D)}^{j\ell}$ and $R_{HJ}^{j\ell} = Q_{HJ,(D)}^{j\ell} + R_{HJ,(D)}^{j\ell}$.  Combining \eqref{eqQHJDbds}, \eqref{eq:suppparamCt}, \eqref{eq:suppContain}, and \eqref{eq:coBdRemainderSol}  into \eqref{eq:RTHdef}, we obtain the estimate
\ali{
\co{R_T} + \co{R_H} &\lsm B_\la^{-1} \plhxi^{5/2} \fr{e_v^{1/2} e_R^{1/2}}{N}. \label{eq:finallyBddRHT}
}
To sum the estimates we have also used the fact that the number of distinct cylinders of the form \eqref{eq:suppContain} that can intersect at a given point in space-time $(t,x)$ is bounded by an absolute constant.  To check this fact, note that if two cylinders indexed by $J$ and $J'$ intersect at a point $(t^*, x^*) \in \R \times \T^3$, then
\ali{
(t^*, x^*) \in \ECyl(\th, \Xi^{-1}; t(I), x_0(J) ) &\cap \ECyl(\th, \Xi^{-1}; t(I'), x_0(J') )  \\
\Rightarrow I = I' \mbox{ and } \quad (t(I), x_0(J)), (t(I), x_0(J')) \stackrel{\eqref{eq:duality}}{\in}& \LCyl(\th, \Xi^{-1}; t^*, x^*) \\
(t(I), x_0(J)), (t(I), x_0(J')) \stackrel{\eqref{eq:containCylsequiv}}{\in}& \ECyl(\th, e^{\th \co{\nb v_\ep}} \Xi^{-1}; t^*, x^*) \subseteq \ECyl(\th, 3 \Xi^{-1}; t^*, x^*).
}
The number of indices $J = (I,f)$ for which $(t(I), x_0(J))$ can belong to a given ball of radius $3 \Xi^{-1} \lsm \Pi^{-1}$ is bounded by an absolute constant by the construction of the cutoff functions.

We can now take $B_\la$ to be a sufficiently large number such that the right hand side of \eqref{eq:finallyBddRHT} is bounded by $\plhxi^{5/2} \fr{e_v^{1/2} e_R^{1/2}}{20 N}$ (and so that $\la = B_\la N \Xi \in \Z$ is an integer).  This choice achieves our desired bound for $\co{R_1}$ when combined with \eqref{eq:c0bdRMS}.  The desired bounds for higher derivatives
\ali{
\co{\nb_{\va} R_T} + \co{ \nb_{\va} R_H } &\lsm (N \Xi)^{|\va|} \plhxi^{5/2} \fr{e_v^{1/2} e_R^{1/2}}{N}, \qquad 1 \leq |\va| \leq 3,
}
now follow from \eqref{eqQHJDbds}, \eqref{sec:UTHJDbd}, \eqref{eq:derivRmdrBdSoln} and the observations concerning the overlaps of the cylinders \eqref{eq:suppContain}.  The assertions about the desired support of $R_1^{j\ell}$ asserted in Lemma~\ref{lem:convexInt} are clear from construction.

The proof of Lemma~\ref{lem:convexInt} will now be complete after explaining the proof of Proposition~\ref{prop:parametric}.



\subsection{The Parametrix Expansion} \label{sec:paramExpand}
We now prove Proposition~\ref{prop:parametric} using the argument in the proof of \cite[Proposition 17.6]{isettOnsag}.  Let $U^\ell = u^\ell \om(\la \Ga_I)$ be given as in the assumptions of Proposition~\ref{prop:parametric}.  By Fourier-expanding $\om(X)$ as a function on $\T^3$, we have
\ali{
U^\ell &= \sum_{m \neq 0} \hat{\om}(m) e^{i \la \xi_m(t,x)} u^\ell(t,x) \label{eq:fourierExpand}
}
where $m \in \Z^3$ and $\xi_m(t,x) := m \cdot \Ga_I(t,x)$.  Following the proof of \cite[Proposition 17.6]{isettOnsag}, we set
\ali{
Q_{(D)}^{j\ell} &= \sum_{m \neq 0} \hat{\om}(m) Q_{(D),m}^{j\ell}, \qquad
Q^{j\ell}_{(D), m} := \la^{-1} \sum_{k=1}^D e^{i \la \xi_m} q_{(k),m}^{j\ell}. \label{eq:parametrix} 
}
The amplitudes $q_{(k),m}^{j\ell}$ are constructed inductively with a sequence of amplitudes $u_{(k),m}^\ell$ such that
\ali{
\begin{split}
i \nb_j \xi_m q_{(k),m}^{j\ell} &= u_{(k-1),m}^\ell \label{eq:algEqnforqjamp}\\
u_{(k),m}^\ell &= - \la^{-1} \nb_j q_{(k),m}^{j\ell}
\end{split}
}
and $u_{(0),m}^\ell = u^\ell$.  By \eqref{eq:fourierExpand},\eqref{eq:algEqnforqjamp} and induction on $D$, we then obtain
\ali{
\begin{split}
U^\ell &= \nb_j Q_{(D)}^{j\ell} + U_{(D)}^\ell \\
U_{(D)}^\ell &= \sum_{m \neq 0} \hat{\om}(m) e^{i \la \xi_m} u_{(D),m}^\ell. \label{eq:UDform}
\end{split}
}
More specifically, to solve \eqref{eq:algEqnforqjamp} we first choose smooth functions $\bar{q}_a^{j\ell}(p)$ of a variable $p \in \R^3 \setminus \{0\}$, symmetric in $j\ell$, such that each $\bar{q}_a^{j\ell}(p)$ is degree $-1$ homogeneous ($\bar{q}_a^{j\ell}(\a p) = \a^{-1} \bar{q}_a^{j\ell}(p)$ if $\a \in \R \setminus \{0\}$) and such that $i p_j \bar{q}_a^{j\ell}(p) = \de_a^\ell$ for all $p \neq 0$.  See \cite[Proposition 17.6]{isettOnsag} for an explicit example.  We then set $q_{(k),m}^{j\ell} := \bar{q}_a^{j\ell}(\nb \xi_m) u_{(k-1),m}^a$, so that \eqref{eq:algEqnforqjamp} is satisfied.  

From this construction we see that both $Q_{(D)}^\ell$ and $U_{(D)}^\ell$ have supported contained in $\supp u^\ell$.  We obtain the desired estimates for $Q_{(D)}^\ell$ and $U_{(D)}^\ell$ stated in Proposition~\ref{prop:parametric} from the formulas \eqref{eq:parametrix}  and \eqref{eq:UDform} by using the following bounds 
\ali{
\co{\nb_{\va} q_{(k),m}^{j\ell} } &\lsm N^{-(k-1)/2} N^{|\va|/2} \Xi^{|\va|} H_{D+3,0}[u], \quad \tx{ for all } 0 \leq |\va| \leq D - k + 4, \quad 1 \leq k \leq D \label{eq:qkparambdsiterate}\\
\co{\nb_{\va} u_{(k),m}^\ell } &\lsm B_\la^{-1} N^{-k/2} N^{|\va|/2} \Xi^{|\va|} H_{D+3,0}[u] \quad \tx{ for all } 0 \leq |\va| \leq D - k + 3, \quad 1 \leq k \leq D \label{eq:ampseqParambds} 
}
from the proof of \cite[Proposition 17.6]{isettOnsag} 
(where $H_{D+3,0}[u]$ is written simply as $H$), 
and by using the rapid decay of $|\hat{\om}(m)| \lsm (1+|m|)^{-40}$ to ensure convergence in the summation over $m \in \Z^3$.  (The main point in the estimate is that each spatial derivative of the sum costs at most a factor of $\la$.)


\section{Iterating the Main Lemma} \label{sec:iterateMainLem}
We now explain the proof of Theorem~\ref{thm:mainThm}.  Similar to other convex integration constructions, the theorem will be proven by repeatedly applying Lemma~\ref{lem:mainLem} to obtain a sequence of Euler-Reynolds flows $(v_{(k)}, p_{(k)}, R_{(k)})$ indexed by $k$ (with frequency energy levels bounded by $(\Xi_{(k)}, e_{v,(k)}, e_{R,(k)})$) that will converge uniformly to the solution $v$ stated in Theorem~\ref{thm:mainThm}.  Unlike previous works, we introduce here a new and sharper approach to estimating the regularity and to optimizing the choice of parameters governing the growth of frequencies.


To initialize the construction, we construct a smooth Euler-Reynolds flow $(v_{(1)}, p_{(1)}, R_{(1)})$ with compact support in time that satisfies
\ali{
\sup_{x \in \T^3} v_{(1)}(0, x) &\geq 10 \label{eq:nontriv0time}
}
and has frequency-energy levels (to order $3$ in $C^0$) bounded by $(\Xi_{(1)}, e_{R,(1)}, e_{R,(1)})$, where $\Xi_{(1)} = \hxi_{(1)}$ and $e_{R,(1)}$ are respectively large and small parameters that remain to be chosen.  One way to produce such an Euler-Reynolds flow is to apply the Main Lemma in the convex integration scheme of \cite{isett} (as was done in \cite{isettOnsag}).  This approach has some added benefits such as the ability arbitrarily large increases in energy within an arbitrarily small time interval \cite{isett}.  For the present purpose it will suffice to take a simpler approach.

We take $v_{(1)}$ to have the form $v_{(1)}^\ell = \psi(B^{-1} t) U^\ell$, where $\psi$ be a smooth cutoff with $\psi(0) = 1$ and $0 \leq \psi(t) \leq 1$ for all $t$, $B$ is a large parameter, and $U^\ell : \T^3 \to \R^3$ is a smooth vector field that satisfies 
\ali{
\int_{\T^3} U^\ell(x) dx = 0, \qquad \nb_\ell U^\ell = 0, \qquad \nb_j(U^j U^\ell) = 0, \qquad \sup_{x \in \T^3} U^\ell(x) \geq 10. \label{eq:needUlsample}
}
For example, one can take a sufficiently large Mikado flow for $U^\ell(x)$.  We then take $p_{(1)} = 0$ and $R_{(1)}$ to be a symmetric tensor that solves
\ali{
\nb_j R_{(1)}^{j\ell} &= \pr_t v_{(1)}^\ell = B^{-1} \psi'(B^{-1} t) U^\ell(X) \label{eq:divEqnR1st}
}
by applying an appropriate, degree $-1$ Fourier-multiplier to the right-hand side of \eqref{eq:divEqnR1st}.  The Euler-Reynolds flow $(v_{(1)}, p_{(1)}, R_{(1)})$ obtained in this way has frequency energy levels (to order $3$ in $C^0$) bounded by $(\overline{\Xi}, 1, e_{R,(1)} )$, where $\overline{\Xi}$ depends only on $U^\ell$, and where $e_{R,(1)} \lsm B^{-1}$ can be made arbitrarily small by taking $B$ large depending on $U^\ell$.  It follows from Definition~\ref{defn:frenlvls} that $(v_{(1)}, p_{(1)}, R_{(1)})$ also have frequency energy levels bounded by $(\Xi_{(1)}, e_{v,(1)}, e_{R,(1)}) := (\overline{\Xi} e_{R,(1)}^{-1/2}, e_{R,(1)}, e_{R,(1)} )$, where we have now fixed our choice of $\Xi_{(1)} := \overline{\Xi} e_{R,(1)}^{-1/2}$ in terms of the small parameter $e_{R,(1)}$ that remains to be chosen.


\subsection{Heuristics and deriving the optimization problem for the parameters} \label{sec:heuristics}

The sequence of frequency-energy levels $(\Xi, e_v, e_R)_{(k)}$ and Euler-Reynolds flows will now be determined by repeatedly applying Lemma~\ref{lem:mainLem} so that the following rules hold.  (Here $\hc$ and $C_L$ denote the two constants of Lemma~\ref{lem:mainLem} and $\hxi_{(k)} := (e_v/e_R)_{(k)}^{1/2} \Xi_{(k)}$.)
\ali{
\Xi_{(k+1)} &= \hc N_{(k)} \Xi_{(k)} \label{eq:Xik1} \\
e_{v,(k+1)} &= (\log \hxi_{(k)}) e_{R,(k)} \\
e_{R,(k+1)} &= \fr{e_{R,(k)}}{g_{(k)}} \\
N_{(k)} &= (\log \hxi_{(k)})^{A} \ever_{(k)}^{1/2} g_{(k)}, \qquad A := 5/2. \label{eq:nkdef}
}
The sequence $g_{(k)} > 1$ describes the ``gain'' in the size of the error after stage $k$, and the sequence of frequency growth parameters $N_{(k)}$ is determined by inequality \eqref{eq:newFreqEn} in Lemma~\ref{lem:mainLem} so that this choice of $N_{(k)}$ achieves the desired gain.  
To work with the estimate \eqref{ineq:coBdV}, it will also be useful to impose that
\ali{
(\log \hxi_{(k+1)})^{1/2} e_{R,(k+1)}^{1/2} \leq \fr{1}{2} (\log \hxi_{(k)})^{1/2} e_{R,(k)}^{1/2}, \mbox{ for all } k \geq 1. \label{eq:estimateXis}
}
The Euler-Reynolds flows constructed by repeatedly applying Lemma~\ref{lem:mainLem} using the above choice of parameters $N_{(k)}$ will converge uniformly to the velocity field $v^\ell = v_{(1)}^\ell + \sum_{k = 1}^\infty V_{(k)}^\ell$. Assuming \eqref{eq:estimateXis}, which is verified in Proposition~\ref{prop:gotShrinking} below, this solution will be nontrivial and continuous for $e_{R,(1)}$ chosen small enough (depending on $\overline{\Xi}$, $\hc$ and $C_L$) thanks to \eqref{eq:nontriv0time} and 
\ali{
\sum_{k = 1}^\infty \co{V_{(k)}^\ell} &\stackrel{\eqref{ineq:coBdV},\eqref{eq:estimateXis}}{\leq} \sum_{k = 0}^\infty C_L (\log \hxi_{(1)})^{1/2} e_{R,(1)} 2^{-k} \leq 5. \label{eq:nontrivBound}
}
As $R_{(k)}$ converges uniformly to $0$, one has from the Euler-Reynolds system that the associated sequence of pressures $p_{(k)} = \De^{-1} \nb_j \nb_\ell(R_{(k)}^{j\ell} - v_{(k)}^j v_{(k)}^\ell)$ converge weakly in $\DD'(\R \times \T^3)$ to $p = -\De^{-1} \nb_j \nb_\ell(v^j v^\ell)$, and that the pair $(v,p)$ form a weak solution to the Euler equations.

Our goal is now to choose $g_{(k)}$ that optimize the regularity of the solution $v$.  The key evolution rule that isolates $1/3$ as the limiting regularity and plays a key role in our analysis will be the following:
\ali{
\de_{(k)} \left( \fr{1}{3} \log \hxi_{(k)} + \fr{1}{2} \log e_{R,(k)} \right) &= \left(\fr{A}{3} + \fr{1}{6}\right) \log \log \hxi_{(k)} + \log \hc. \label{eq:keyEvolRule}
}
Here $\de_{(k)}[f_{(k)}] = f_{(k+1)} - f_{(k)}$ is the discrete differencing operator and $A = \fr{5}{2}$.  A crucial point is that \eqref{eq:keyEvolRule} holds for {\it all} possible choices of $g_{(k)}$.

With the goal of computing regularity in mind, suppose $\De x \in \R^3$ with, say, $0 < |\De x| \leq 10^{-2}$.  Writing $v = v_{\brk} + \sum_{k\geq \bk} V_{(k)} $ and $L_k := \log \hxi_{(k)}$, we can bound $|v(t,x+\De x) - v(t,x)|$ using \eqref{eq:estimateXis} by
\ali{
|v(t,x+\De x) - v(t,x)| &\leq \co{\nb v_{\brk}} |\De x| + \sum_{k \geq \bk} 2 \co{ V_{(k)}} \notag \\
&\leq \Xi_{\brk} e_{v,\brk}^{1/2} |\De x| + 4 C_L (\log \hxi_{\brk})^{1/2} e_{R,\brk}^{1/2} \notag \\
|v(t,x+\De x) - v(t,x)|&\leq 4 C_L L_{\bk} \left(\hxi_{\brk} |\De x| + 1 \right) e_{R,\brk}^{1/2}. \label{eq:regEstimate}
}
The estimate is optimized by choosing $\bk$ to be the largest value $k$ for which $\hxi_{(k)} |\De x| \leq 1$.  Now assuming $\bk$ has been chosen as this value, the estimate \eqref{eq:regEstimate} leads to
\ali{
|v(t,x+\De x) - v(t,x)| &\leq 8 C_L L_{\bk} e_{R,\brk}^{1/2} = 8 C_LL_{\bk} \hxi_{\brk}^{-1/3} \mbox{exp}\left(\fr{1}{3} \log \hxi_{\brk} + \fr{1}{2} \log e_{R,\brk} \right) \notag \\
&\lsm L_{\bk} \hxi_{(\bk+1)}^{-1/3} \mbox{exp}\left(\fr{1}{3}\de_{(k)} \log \hxi_{(k)}\Big|_{k = \bk} + \fr{1}{3} \log \hxi_{(k)} + \fr{1}{2} \log e_{R,(k)} \right) \notag \\
|v(t,x+\De x) - v(t,x)| &\lsm |\De x|^{1/3} L_{\bk} \mbox{exp}\left(\fr{1}{3}\de_{(k)} \log \hxi_{(k)}\Big|_{k = \bk} + \fr{1}{3} \log \hxi_{(k)} + \fr{1}{2} \log e_{R,(k)} \right). \label{eq:wantToMin} 
}
Using \eqref{eq:keyEvolRule} to expand $\fr{1}{3} \log \hxi_{(k)} + \fr{1}{2} \log e_{R,(k)}$, we minimize the right hand side of \eqref{eq:wantToMin} if we minimize
\ali{
H_{\bk} := \left(\fr{1}{3}(\log \hxi_{(\bk+1)} - \log \hxi_{(\bk)} ) + \sum_{k=1}^{\bk-1} (\log \log \hxi_{(k)} + \log \hc) \right). \label{eq:balanceThese}
}
The expression \eqref{eq:balanceThese} now reveals the optimization problem for choosing $g_{(k)}$.  Namely, to control the term $\de_{(k)} \log \hxi_{(k)}$ the frequencies should not grow too quickly.  However, a slow growth of frequencies produces a long summation and a poor estimate for the sum as the construction is iterated many times before achieving a given length scale.  Intuitively, the best estimate should be achieved if the two terms are balanced, which suggests the parameters $L_k = \log \hxi_{(k)}$ should satisfy the discrete version of the equation $\fr{dL}{dk} = 3 \int_1^k (\log L(\kappa) + c) d\kappa$, whose solutions grow like $L_k = \left(3 + o(1)\right)k^2\log k$ at infinity.

\subsection{Parameter asymptotics and optimization}

With this motivation, we take $g_{(k)} = e^{\ga k \log k}$, where $\ga > 0$ is a parameter that will be chosen to optimize the regularity.  To simplify the algebra we can restrict to $k \geq 2$ by assuming that the Euler-Reynolds flows $(v_{(1)}, p_{(1)}, R_{(1)}) = (v_{(2)}, p_{(2)}, R_{(2)})$ and their frequency energy levels are equal.  

Before estimating the regularity, we wish to fix our choice of the parameter $e_{R,(1)}$ that dictates the initial frequency energy levels.  We therefore verify the assumption \eqref{eq:estimateXis} (restricting now to $\ga \geq 2$).
\begin{prop} \label{prop:gotShrinking} If $\ga \geq 2$ and $e_{R,(1)}$ is small enough depending on $\hc$, then \eqref{eq:estimateXis} holds for all $k \geq 2$.
\end{prop}
\begin{proof} Taking logs of \eqref{eq:estimateXis}, it suffices to bound the quantity
\ali{
\fr{1}{2} \de_{(k)} \log \log \hxi_{(k)} + \fr{1}{2} \de_{(k)} \log e_{R,(k)} &= \fr{1}{2} \de_{(k)} \log \log \hxi_{(k)} - \fr{1}{2} \log g_{(k)} \label{eq:gottaBeShrinking} 
}
by $- \log 2$ uniformly in $k$.

Towards this goal, we set $Z_k := \hc(\log \hxi_{(k)})^{A +1/2}$ to be the lower order factor from \eqref{eq:Xik1}, \eqref{eq:nkdef}. Linearizing $\log(\cdot)$ around $L_k := \log \hxi_{(k)}$ and using \eqref{eq:Xik1}-\eqref{eq:nkdef} and concavity, we have 
\ali{
\de_{(k)} \log \log \hxi_{(k)} &= \log \left( \log \hxi_{(k)} + \log (Z_k g_{(k)}^{3/2} ) \right) - \log \log \hxi_{(k)} \notag \\
\de_{(k)} \log \log \hxi_{(k)} &\leq \fr{\log (Z_k g_{(k)}^{3/2} )}{\log \hxi_{(k)}}.  \label{eq:loglogGrowth} 
}
We now substitute \eqref{eq:loglogGrowth} into \eqref{eq:gottaBeShrinking} and take $e_{R,(1)}$ small to ensure that $\hxi_{(k)} \geq \Xi_{(k)} \geq \Xi_{(1)} = \overline{\Xi} e_{R,(1)}^{-1/2}$ is large enough so that the following bound holds for all $k \geq 2$
\ali{ 
\eqref{eq:gottaBeShrinking} &\leq \fr{ \log Z_k }{\log \hxi_{(k)}} - \fr{1}{3} \log g_{(k)}. \label{eq:intermShrinking}
}
Taking $e_{R,(1)}$ smaller and hence $\Xi_{(1)}$ larger, we can ensure that the function $f(\Xi) := \fr{\log(\hc(\log \Xi)^{A+1/2})}{\log \Xi}$ is decreasing in $\Xi$ on the interval $\Xi \in [\Xi_{(1)}, \infty)$.  From $\hxi_{(k)} \geq \Xi_{(1)}$ and \eqref{eq:intermShrinking} we obtain
\ali{
\eqref{eq:intermShrinking} &\leq \fr{\log( \hc (\log \Xi_{(1)})^{A + 1/2} )}{\log \Xi_{(1)}} - \fr{1}{3} \log g_{(2)}, \qquad \mbox{ for all } k \geq 2. \label{eq:timeToChooseer1}
}
We have that $- (1/3) \log g_{(2)} = -(2\ga/3) \log 2 \leq - (4/3) \log 2$.  Taking $e_{R,(1)}$ small and thus $\Xi_{(1)}$ large, we can bound \eqref{eq:timeToChooseer1} and therefore \eqref{eq:gottaBeShrinking} by $- \log 2$, which establishes Proposition~\ref{prop:gotShrinking}.
\end{proof}
At this point, we choose $e_{R,(1)}$ sufficiently small (depending on $\hc$ and $C_L$) to satisfy the assumptions of Proposition~\ref{prop:gotShrinking} and such that \eqref{eq:nontrivBound} holds.

With the initial frequency energy levels determined, we now turn to the asymptotics of the frequency energy levels for large $k$.  These asymptotics are summarized as follows.
\begin{prop} For all $k \geq 3$ and the above choice of $g_{(k)}$, we have the following asymptotics
\ali{
- \log e_{R,(k)} &= \fr{\ga k^2}{2} \log k + O(k \log k) \label{eq:eRkasympt} \\
\fr{1}{2} \log \ever_{(k)} &= \fr{1}{2} \ga k \log k + O(\log k) \label{eq:everkAsymp} \\
\de_{(k)} \log \hxi_{(k)} &= \fr{3}{2} \ga k \log k + O(\log k) \label{eq:freqChangeRate} \\
\log \hxi_{(k)} &= \fr{3}{2} \fr{\ga k^2}{2} \log k + O(k \log k) \label{eq:loghxiasymp} \\
\log \log \hxi_{(k)} &= 2 \log k + O(1) \label{eq:loglogasymp}\\
\fr{1}{3} \log \hxi_{(k)} + \fr{1}{2} \log e_{R,(k)} &= 2 \left(\fr{A}{3} + \fr{1}{6} \right) k \log k + O(k) \label{eq:loghxierkasmp}
}
together with the following bounds
\ali{
(\log \hxi_{(k)})^{-1} &= O(k^{-2} (\log k)^{-1}) \label{eq:lowerFreqBd}\\
\log \log \hxi_{(k)} &= O(\log k). \label{eq:loglogBd} 
}
Here the implicit constants in the $O(\cdot)$ notation depend only on $\hc, \ga$, $\Xi_{(1)}, e_{R,(1)}$, and $A = 5/2$.
\end{prop}
The proof will proceed by induction on $k \geq 3$ and will use some extra notation for the induction.  We write $C_{\eqref{eq:eRkasympt}}, \ldots,C_{\eqref{eq:loglogBd}}$ to refer to the implicit constants in the Big-$O$ notation in the Proposition.  For example the term in \eqref{eq:loglogBd} is bounded by $|O(\log k)| \leq C_{\eqref{eq:loglogBd}} \log k$.  
We assume at the onset that all the constants $C_{\eqref{eq:eRkasympt}}, \ldots,C_{\eqref{eq:loglogBd}}$ are sufficiently large depending on $\Xi_{(1)}$ and $e_{R,(1)} = e_{v,(1)}$ such that the bounds \eqref{eq:eRkasympt}-\eqref{eq:loglogBd} hold for $k = 3$.  The proof will make use of the Taylor expansion formula
\ali{
f(X + Y) = f(X) + Y \int_0^1 f'(X + \si Y) d\si = f(X) + f'(X) Y + Y^2 \int_0^1 (1-\si) f''(X + \si Y) d\si. \label{eq:taylorExpand}
}
\begin{proof}[Proof of \eqref{eq:eRkasympt}] The equality follows from the evolution rule $\log e_{R,(k+1)} = - \log g_{(k)} + \log e_{R,(k)}$ and
\ali{
\sum_{1 \leq I \leq k} \log g_{(I)} &= \sum_{1 \leq I \leq k} \ga I \log I = \fr{\ga k^2}{2} \log k + O(k \log k), \qquad k \geq 3
}
(where the constant above depends on $\ga$).
\end{proof}
\begin{proof}[Proof of \eqref{eq:lowerFreqBd}]  From $\log \hxi_{(k+1)} \geq \log g_{(k)} + \log \hxi_{(k)}$, we have $k^2 \log k \lsm \sum_{3 \leq I \leq k} \log g_{(I)}  \leq \log \hxi_{(k)}.$
\end{proof}

\begin{proof}[Proof of \eqref{eq:loglogBd}]  Let $L_k := \log \hxi_{(k)}$ and $Z_k = \hc(\log \hxi_{(k)})^{A +1/2}$.  Then for some $A_0 \geq 1$ and all $k \geq 3$,
\ali{
\de_{(k)} \log \log \hxi_{(k)} &= \log ( L_k + \log(Z_k g_{(k)}^{3/2}) ) - \log L_k \notag \\
\de_{(k)} \log \log \hxi_{(k)} &\leq L_k^{-1} (\log Z_k + \log g_{(k)}^{3/2} ) \leq A_0 C_{\eqref{eq:lowerFreqBd}} (k^{-2} (\log k)^{-1} \log \log \hxi_{(k)} + k^{-1}).
}
Choose $k^* = k^*(C_{\eqref{eq:lowerFreqBd}})$ large so that $A_0 C_{\eqref{eq:lowerFreqBd}}k^{-2} \leq 10^{-1} \de_{(k)} \log k$ for all $k \geq k^*$ and assume that $C_{\eqref{eq:loglogBd}}$ is large enough so that \eqref{eq:loglogBd} holds for $k \leq k^*$.  

We now proceed by induction on $k$ to obtain \eqref{eq:loglogBd} for $k > k^*$.  Assuming \eqref{eq:loglogBd} for $k$, we have  
\ali{
\de_{(k)} \log \log \hxi_{(k)}&\leq 10^{-1} C_{\eqref{eq:loglogBd}} \de_{(k)} \log k  +  A_0 C_{\eqref{eq:lowerFreqBd}} k^{-1} \leq C_{\eqref{eq:loglogBd}} \de_{(k)} \log k, \qquad \mbox{ for } k \geq k^* \label{eq:boundInChange}
}
if $C_{\eqref{eq:loglogBd}}$ is sufficiently large, which implies \eqref{eq:loglogBd} for $k + 1$, and thus for all $k \geq k^*$ by induction.
\end{proof}
\begin{proof}[Proof of \eqref{eq:everkAsymp}]  The equality follows from \eqref{eq:loglogBd} and
$\fr{1}{2} \log (e_v/e_R)_{(k+1)} = \fr{1}{2} (\log g_{(k)} + \log \log \hxi_{(k)})$.
\end{proof}
\begin{proof}[Proof of \eqref{eq:freqChangeRate}-\eqref{eq:loghxiasymp}]  For $k \geq 3$, we have by \eqref{eq:everkAsymp} and \eqref{eq:loglogBd} (for $A = 5/2$)
\ALI{
\de_{(k)} \log \hxi_{(k)} &= \fr{1}{2} \log (e_v/e_R)_{(k+1)} + \log g_{(k)} + A \log \log \hxi_{(k)} \\
&= \fr{3\ga}{2} k \log k + O(\log k) = \fr{3\ga}{2} \de_{(k)}\left[\fr{k^2}{2} \log k \right] + O(\de_{(k)}[k \log k]),
}
which implies both \eqref{eq:freqChangeRate} and \eqref{eq:loghxiasymp} after summing over $k$.  
\end{proof}
\begin{proof}[Proof of \eqref{eq:loglogasymp}] Again writing $L_k = \log \hxi_{(k)}$ and $Z_k = \hc(\log \hxi_{(k)})^{A +1/2}$, we have by Taylor expansion
\ALI{
\de_{(k)} \log \log \hxi_{(k)} &= \log(L_k + \log(Z_k g_{(k)}^{3/2}) ) - \log L_k \\
&= (\log \hxi_{(k)})^{-1} \log(Z_k g_{(k)}^{3/2}) - \int_0^1 d\si \fr{(\log(Z_k g_{(k)}^{3/2}))^2(1-\si)}{\left(L_k + \si \log(Z_k g_{(k)}^{3/2}) \right)^2}.  
}
The main term is $(\log \hxi_{(k)})^{-1} \log g_{(k)}^{3/2} = 2 k^{-1} + O(k^{-2}) = 2\de_{(k)} \log k + O(k^{-2})$ by \eqref{eq:loghxiasymp}.  The remaining terms are of size $O(k^{-2})$ by \eqref{eq:loglogBd} and \eqref{eq:loghxiasymp} again.  Summing over $k$ gives \eqref{eq:loglogasymp}.
\end{proof}

\begin{proof}[Proof of \eqref{eq:loghxierkasmp}]  Equation \eqref{eq:loghxierkasmp} follows from \eqref{eq:keyEvolRule}, \eqref{eq:loglogasymp} and summation over $k$.
\end{proof}

We now return to analyzing the regularity estimate \eqref{eq:wantToMin}.  From \eqref{eq:loghxiasymp}, \eqref{eq:freqChangeRate}, \eqref{eq:loghxierkasmp}, and by the definitions of $\bk$ and $\hxi_{\brk}$, we obtain (using \eqref{eq:taylorExpand} with $f(X) = X^{-1}$ or $\log X$) that for all $|\De x| \leq 10^{-2}$
\ali{
\bk^2 \log \bk \lsm  \log \hxi_{\brk} &\leq \log |\De x|^{-1} \leq \log \hxi_{(\bk + 1)} \lsm \bk^2 \log \bk \notag \\
\fr{3\ga}{4} \bk^2 \log \bk &= \log |\De x|^{-1} + O(\bk \log \bk) \label{eq:kbarlogdxrel} \\
(\log |\De x|^{-1})^{-1} &= \left(\fr{4}{3 \ga} + O(\bk^{-1})\right) \bk^{-2} (\log \bk)^{-1} \notag \\
\log(\bk^2) &= \log \log |\De x|^{-1} + O(\log \log \bk). \label{eq:logkloglogdex} 
}
To bound \eqref{eq:wantToMin} purely in terms of $|\De x|$, we first estimate the logarithm of the term $L_{\bk} \mbox{exp}(H_{\bk})$ appearing in \eqref{eq:wantToMin}-\eqref{eq:balanceThese} (using $A = 5/2$ and $\fr{A}{3} + \fr{1}{6} = 1$) by
\ali{
(\log |\De x|^{-1})^{-1} \cdot (H_{\bk} + \log L_{\bk}) &= \left(\fr{4}{3\ga} \left( 1 + O(\bk^{-1})\right) (\bk^2 \log \bk)^{-1} \right) \cdot \left(\left(\fr{\ga}{2} + 2 \right) \bk \log \bk + O(\bk) \right) \notag
\\
&= \fr{4}{3\ga}\left(\fr{\ga}{2} + 2\right) \bk^{-1}  + O(\bk^{-1}(\log \bk)^{-1}) \notag \\
&= \fr{4}{3\ga}\left(\fr{\ga}{2} + 2\right) (\bk^2 \log \bk)^{-1/2} (\log \bk)^{1/2} + O(\bk^{-1}(\log \bk)^{-1}) \notag \\
&= 2^{-1/2}\left(\fr{4}{3\ga}\right)\left(\fr{\ga}{2} + 2\right) (\bk^2 \log \bk)^{-1/2}(\log \log |\De x|^{-1})^{1/2} \notag \\
&+ O\left(\fr{\log \log \bk}{(\bk^2 \log \bk)^{1/2}(\log \log |\De x|^{-1})^{1/2}} \right). \notag
}
In the last line we used \eqref{eq:logkloglogdex} and \eqref{eq:taylorExpand} with $f(X) = X^{1/2}$.  From \eqref{eq:kbarlogdxrel} and \eqref{eq:taylorExpand} we then have 
\ali{
(\log |\De x|^{-1})^{-1} \cdot (H_{\bk} + \log L_{\bk}) &= 2^{-1/2} \left( \fr{4}{3 \ga} \right)^{1/2} \left(\fr{\ga}{2} + 2\right) (\log |\De x|^{-1})^{-1/2} (\log \log |\De x|^{-1})^{1/2} \label{eq:timeToOptimize}\\
&+ O\left(\fr{\log \log \log |\De x|^{-1}}{(\log |\De x|^{-1})^{1/2}(\log \log |\De x|^{-1})^{1/2}} \right). \label{eq:messyLowerOrder}
} 
The bound \eqref{eq:timeToOptimize} is optimized by taking $\ga = 4$, which is precisely the value that leads to the asymptotic $\log \hxi_{(k)} = (3 + o(1)) k^2 \log k$ predicted by the heuristics at the conclusion of Section~\ref{sec:heuristics}.  Substituting into \eqref{eq:wantToMin}, we finally obtain
\ali{
|v(t,x + \De x) - v(t,x)| &\lsm |\De x|^{1/3 - B \sqrt{\fr{\log \log |\De x|^{-1}}{\log |\De x|^{-1} } } }, \label{eq:finalEndpointReg}
}
where one can take the constant $B = 2\sqrt{2/3}$ at the expense of introducing the additional lower order term\footnote{The derivation of equation \eqref{eq:balanceThese} suggets that taking $g_{(k)} = (\sum_{I=1}^k (\log \log \hxi_{(I)} + \log \hc)) + (\log \log \hxi_{k}/2)$ would optimize the lower order terms as well, although this alternative choice would not affect the leading order terms. } from \eqref{eq:messyLowerOrder}.  In particular, $v$ belongs to $v \in \bigcap_{\a < 1/3} L_t^\infty C_x^\a$, and therefore belongs to $v \in \bigcap_{\a < 1/3} C_{t,x}^\a$ by the results in \cite{isett2}.  To check that $v$ has compact support in time, note that the time support in each iteration grows by at most a factor 
\ALI{
\Xi^{-1}_{(k)} e_{v,(k)}^{-1/2} = \hxi_{(k)}^{-1} e_{R,(k)}^{-1/2} = \hxi_{(k)}^{-2/3} \mbox{exp}\left(- \fr{1}{3} \log \hxi_{(k)} - \fr{1}{2} \log e_{R,(k)}\right).
}
Using \eqref{eq:loghxiasymp} and \eqref{eq:loghxierkasmp}, we conclude that the series $\sum_k \Xi_{(k)}^{-1} e_{v,(k)}^{-1/2}$ converges, and hence the limiting solution is supported on a finite time interval.  This calculation concludes the proof of Theorem~\ref{thm:mainThm}.

\section{Improving the Borderline Estimate} \label{sec:improveBorderline}
In this section, we sketch roughly how the value of the $B$ appearing in the regularity estimate \eqref{eq:finalEndpointReg} can be improved by combining with the approach to the Gluing Lemma introduced in \cite{BDLSVonsag}.  

Recall that, in the notation of \cite{isettOnsag}, the Gluing Lemma is proved by introducing for a given Euler-Reynolds flow $(v, p, R)$ corrections $y^\ell = \sum_I \eta_I y_I^\ell$ and $\bp = \sum_I \eta_I \bp_I$ to the velocity and pressure
such that the new velocity field $\tilde{v}^\ell = v^\ell + y^\ell$ and pressure $\tilde{p} = p + \bp$ solve the Euler-Reynolds system with a new Reynolds stress $\wtld{R}$ that is supported in disjoint time intervals of width $\th \sim \plhxi^{-2} \Xi^{-1} e_v^{-1/2}$.  The new stress $\wtld{R}$ is constructed in terms of symmetric tensors $r_I^{j\ell}$ that solve $\nb_j r_I^{j\ell} = y_I^\ell$, which are obtained by solving the following initial value problem\footnote{Here we have simplified the equations by combining the equations for the $\rho_I^{j\ell}$ and $z_I^{j\ell}$ from \cite{isettOnsag} into one equation.}:
\ali{
\begin{split}
(\pr_t + v^i \nb_i) r_I^{j\ell} &=  \RR^{j\ell}[\nb_i [\nb_ av^i r_I^{ab}] - y_I^i \nb_i v^b] - y_I^j y_I^\ell - \bp_I \de^{j\ell} - R^{j\ell} \label{eq:eqnForrI} \\
r_I^{j\ell}(t(I), x) &= 0.
\end{split}
}
Here $\RR^{j\ell}$ is an order $-1$ operator that inverts the divergence equation in symmetric tensors, and the identity $\nb_j r_I^{j\ell} = y_I^\ell$ can be checked using the equation
\ali{
\begin{split}
\pr_t y_I^\ell + v^i \nb_i y_I^\ell + y_I^i \nb_i v^\ell + \nb_j (y_I^j y_I^\ell) + \nb^\ell \bp_I &= - \nb_j R^{j\ell} \\
\pr_t y_I^\ell + v^i \nb_i y_I^\ell + y_I^i \nb_i u_I^\ell + \nb^\ell \bp_I &= - \nb_j R^{j\ell}, \label{eq:yIeqn}
\end{split}
}
where $u_I^\ell = v^\ell + y_I^\ell$ is the classical solution to incompressible Euler with initial data $v^\ell(t_0(I),x)$.

In \cite{BDLSVonsag}, a different approach to solving and estimating solutions of the equation $\nb_j r_I^{j\ell} = y_I^\ell$ is taken.  There, one first considers the potential $\tilde{z}_I = \De^{-1} \nb \times y_I$, which solves $\nb \times \tilde{z}_I = y_I$, $\mbox{ div } \tilde{z}_I = 0$ and turns out to satisfy an evolution equation that (like \eqref{eq:eqnForrI}) has a good structure.  From $\tilde{z}_I$, one then obtains a symmetric anti-divergence for $y_I$ by applying a zeroth order operator (e.g. $r_I^{j\ell} = \RR^{j\ell}[\nb \times \tilde{z}_I$]), which is estimated using Schauder and commutator estimates for CZOs.  (We note that, conversely, estimates for $\tilde{z}_I$ can be deduced from those of $r_I^{j\ell}$ above by similar zeroth order commutator estimates.)  The key simplification comes in treating the term $\De^{-1} \nb \times[ y_I \cdot \nb v ]$ that is analogous to the term $\RR^{j\ell}[y_I \cdot \nb v]$ in \eqref{eq:eqnForrI}, the latter of which had been treated by a decomposition into frequency increments in \cite{isettOnsag}.  For the present applications, the estimates employed in \cite{BDLSVonsag}, which apply the classical local well-posedness theory for Euler and Schauder and commutator estimates for CZOs, are not strong enough as they lose small powers of the frequency $\Xi$, which restricts the regularity to $1/3 - \ep$ for some $\ep > 0$.  
However, as we now explain, combining the techniques in \cite{BDLSVonsag} and \cite{isettOnsag} leads to a logarithmic improvement in the timescale of the gluing and hence a logarithmic improvement in the main estimate of the iteration.

The approach of \cite{BDLSVonsag} can be extended to any dimension using the antisymmetric potential\footnote{We write $\psi^{ab}$ to agree with the usual stream function $\psi$ in dimension $2$, which is related by $\psi^{ab} = \psi \, \ep^{ab}$, where the two-dimensional volume element $\ep^{ab}$ is the unique anti-symmetric tensor with $\ep^{12} = 1$.} defined by $\psi_I^{ab} = \BB^{ab}[y_I] := \De^{-1}(\nb^a y_I^b - \nb^b y_I^a)$, which solves the following Hodge system\footnote{We caution the reader that our normalizations for wedge products are taken to elucidate the present calculations, but do not agree with all standard normalizations, which can differ up to multiplication by constants.}
\ali{
\nb_a \psi^{ab} = y_I^b, \qquad (\nb \wed \psi)^{abc} := \nb^a \psi^{bc} - \nb^b \psi^{ac} + \nb^c \psi^{ab} = 0, \qquad \int_{\T^3} \psi(x) dx = 0. \label{eq:hodgeSystem}
}
Using the anti-symmetry of $\psi_I^{ab}$, one obtains the following identity
\ali{
y_I^i \nb_i v^\ell &= \nb_a [ \psi_I^{ai} \nb_i v^\ell ]. \label{eq:identityWithOm}
}
Using \eqref{eq:identityWithOm} and $\psi_I^{ab} = \BB^{ab}[y_I^j] = \BB^{ab} \nb_i[r_I^{ij}]$ can provide an alternative approach to treating the low frequency part of the term $\RR^{j\ell}[y_I^i \nb_i v^b]$ in \eqref{eq:eqnForrI} and the analogous term in the pressure.

Towards improving the timescale of the gluing, apply \eqref{eq:yIeqn} along with the following calculus identity (which we express in both index and invariant notation)
\ali{
\begin{split}
\De \psi^{ab} &= (\nb \wed [\nb \neg \psi])^{ab} + (\nb \neg [ \nb \wed \psi])^{ab} \\
\nb_i[\nb^i \psi^{ab}] &= \left(\nb^a [ \nb_j \psi^{jb} ] - \nb^b[ \nb_j \psi^{ja} ]\right) + \nb_i [ \nb^i \psi^{ab} + \nb^b \psi^{ia} + \nb^a \psi^{bi}  ] 
\end{split}
}
to derive the following equation for the potential $\psi_I^{ab}$, generalizing\footnote{A slight departure from \cite{BDLSVonsag} is the isolation of quadratic terms of the form $y_I^j y_I^\ell$, which would be estimated jointly in $y_I^i \nb_i v^\ell + \nb_i(y_I^i y_I^\ell) = y_I^i \nb_i u_I^\ell$ in the approach of \cite{BDLSVonsag}.  The $y_I^jy_I^\ell$ terms are kept separate here in order to avoid a resulting additional derivative loss in the estimates. } \cite[Section 3.3]{BDLSVonsag}:
\ali{
\begin{split} \label{eq:omjkeveqn}
\De[(\pr_t + v^i \nb_i) \psi_I^{jk}] =& \nb_a \nb_i[ (\psi_I^{jk} \wed \nb^{a}) v^i ]  - \nb^j \wed[\nb_a( \psi_I^{ai} \nb_i v^k) + \nb_a[ y_I^a y_I^k] + \nb_i R^{ik}  ] \\
&+ \nb^j \wed[\nb_i[ \nb_a v^i \psi_I^{ak}] ], \end{split}
\\
\psi_I^{jk} \wed \nb^a v^i &:= \psi_I^{jk} \nb^a v^i - \psi_I^{ak} \nb^j v^i + \psi_I^{aj} \nb^k v^i. \notag
}
This derivation relies on equation \eqref{eq:hodgeSystem} and $\nb^j \nb^k \bp_I - \nb^k \nb^j \bp_I = 0$, and uses that $\nb_i v^i = 0$ to maintain the divergence form.  The convention above for $\nb^j \wed$ applied to a vector field is $\nb^j \wed u^k := \nb^j u^k - \nb^k u^j$, while $(\psi_I^{jk} \wed \nb^{a}) v^i$ indicates a sum over cyclic permutations of $j k a$ in $\psi_I^{jk} \nb^a v^i$. 

One may now couple equation \eqref{eq:omjkeveqn} to equation \eqref{eq:eqnForrI} while writing $\RR^{j\ell}[ y_I^i \nb_i v^\ell] = \RR^{j\ell}\nb_a[ \psi_I^{ai} \nb_i v^\ell]$ and similarly for the analogous term $\De^{-1} \nb_\ell[ y_I^i \nb_i v^\ell]$ appearing in the pressure $\bp_I$.  By considering a weighted norm $\hh(t) = \hh_I(t)$ such that (setting $\nhat := (e_v/e_R)^{1/2}$ and, for instance, $\a = 1/7$) 
\ali{
\begin{split} \label{eq:familyOfBounds}
\co{ \nb_{\va} r_I} + \co{ \nb_{\va} \psi_I } + \hxi^{-\a}(\cda{\nb_{\va}r_I} + \cda{\nb_{\va}&\psi_I}) \leq \nhat^{(|\va| - 2)_+} \Xi^{|\va|} (\Xi e_v^{1/2})^{-1} e_R \hh(t) \\
\co{ \nb_{\va} y_I} + \hxi^{-\a}\cda{y_I} &\leq \nhat^{(|\va| - 2)_+} \Xi^{|\va|} e_R^{1/2} \hh(t), \quad \mbox{ for } 0 \leq |\va| \leq 3
\end{split}
}
and following the Littlewood-Paley approach to the gluing estimates in \cite{isettOnsag}, we obtain the bound
\ali{
\hh(t) &\lsm \plhxi \Xi e_v^{1/2} \int_0^t ( 1 + \hh(\tau))^2 d\tau. \label{eq:nonGronwall}
}
The prefactor in \eqref{eq:nonGronwall} improves the analogous prefactor in \cite[Proposition 10.1]{isettOnsag} by a factor of $(\log \hxi)^{-1}$, which thus improves the timescale $\th$ by a logarithmic factor to $\th \sim (\log \hxi)^{-1} (\Xi e_v^{1/2})^{-1}$.  What this improvement in timescale yields is that the time cutoff factors of $\eta_I'$ in the terms of the form $\sim \eta_I' r_I^{j\ell}$ that compose the new stress error $\wtld{R}$ have become smaller by a factor $(\log \hxi)^{-1}$ in size, while the anti-divergence terms $r_I^{j\ell}$ have increased in size by a factor of $(\log \hxi)$ over the elongated time scale.

Although the estimate $\co{\wtld{R}} \lsm (\log \hxi) e_R$ on the stress does not improve, the estimate on the advective derivative improves logarithmically to $\co{D_t \wtld{R} } \lsm (\log \hxi)^2 \Xi e_v^{1/2} e_R$.  
The bound \eqref{eq:newFrEnLvls} for the new frequency-energy levels in the Main Lemma similarly improves by one power of $\log \hxi$ to become:
\ali{
(\Xi', e_v', e_R') &= \left(\tilde{C} N \Xi, \plhxi e_R, \plhxi^{A} \fr{e_v^{1/2} e_R^{1/2}}{N} \right), \qquad A = 3/2. \label{eq:improvedLogPower}
}
(One can alternatively pursue an approach closer to \cite{BDLSVonsag} wherein the equation for $\psi_I^{ab}$ is coupled to the evolution equation\footnote{Using the evolution equation for the symmetric anti-divergence is important to avoid an additional logarithmic loss that would be incurred from attempting to deduce estimates for $\tilde{r}_I$ directly from those for $\psi_I$.} for a different, symmetric anti-divergence such as $\tilde{r}_I^{j\ell} := \De^{-1}(\nb^j y_I^\ell + \nb^\ell y_I^j)$.  
Implementing this alternative approach requires additional, sharper commutator estimates.)


The improvement in the power $A = 3/2$ of $\log \hxi$ in \eqref{eq:improvedLogPower} then leads to an improvement in the constant $B$ in the leading order term of the regularity estimate \eqref{eq:finalEndpointReg}.  Namely, repeating the analysis of Section \ref{sec:iterateMainLem} but with $A = 3/2$ instead of $5/2$ improves the leading order term in \eqref{eq:loghxierkasmp}, which leads to a factor of $\left(\fr{\ga}{2} + \fr{4}{3} \right)$ in \eqref{eq:timeToOptimize} in place of $\left( \fr{\ga}{2} + 2 \right)$.  After choosing $\ga = 8/3$ to optimize \eqref{eq:timeToOptimize}, one obtains a leading order constant of $B = 4/3 = 2(2/3)$ instead of $B = 2 \sqrt{2/3}$.  Note that, with the improved constant, the function space implicitly defined by the estimate \eqref{eq:finalEndpointReg} is strictly contained in the one with the larger value of $B$, and the corresponding norms are not comparable to each other.






\bibliographystyle{alpha}
\bibliography{eulerOnRn}

\end{document}